%% file: main.tex
\documentclass[12pt,a4paper,oneside,onecolumn]{article}

\usepackage{libertine}
\usepackage{courier}

\usepackage[english]{babel}

\usepackage{amsthm,amsmath,amssymb}
\usepackage[longnamesfirst,square,numbers,comma,sort&compress]{natbib}
\usepackage{titlesec}
\usepackage{float}
\usepackage{graphicx}
\usepackage{comment}
\usepackage{nicefrac}
\usepackage{datetime}
\usepackage{hyperref}
\usepackage[nameinlink]{cleveref}
\usepackage{bookmark}
\usepackage{color}
\usepackage[shortlabels]{enumitem}
\usepackage{pdfpages}

\usepackage{cancel}
\usepackage{mathtools}
\usepackage{titling}
\usepackage{ulem}

\usepackage{thmtools, thm-restate}

\theoremstyle{plain}

\newtheorem{theorem}{Theorem}[section]
\newtheorem{lemma}[theorem]{Lemma}

\newtheorem{cor}[theorem]{Corollary}
\newtheorem{conj}[theorem]{Conjecture}

\newtheorem{definition}[theorem]{Definition}

\newtheorem{question}[theorem]{Question}

\newtheorem{prop}[theorem]{Proposition}
\theoremstyle{remark}

\newtheorem*{remark*}{Remark}

\newtheorem*{theorem*}{Theorem}
\newtheorem*{lemma*}{Lemma}

\theoremstyle{definition}
\newtheorem*{definition*}{Definition}
\newtheorem*{example*}{Example}

\newtheorem*{prop*}{proposition}
\theoremstyle{remark}

\numberwithin{equation}{section} 

\newcommand{\Ind}{\mathcal{I}^n\big{(}[0,1],\mathbb{R}^d\big{)}}
\newcommand{\In}{\mathcal{I}^{n}\big{(}[0,1],\mathbb{R}\big{)}}
\DeclareMathOperator{\diff}{\mathbf{diff}_{H^n}\big([0,1]\big)}
\DeclareMathOperator{\diffp}{\mathbf{diff}_{H^n}^+\big([0,1]\big)}
\DeclareMathOperator{\diffm}{\mathbf{diff}_{H^n}^-\big([0,1]\big)}
\DeclareMathOperator{\diffO}{\mathbf{diff}^0_{H^n}\big([0,1]\big)}
\DeclareMathOperator{\len}{\mathrm{len}}

\titleformat{\chapter}[hang]{\bf\huge}{\thechapter}{2pc}{}
\titlespacing*{\chapter}{0pt}{-80pt}{30pt}
\title{HUJI M.Sc.~Thesis Template}

\title{Characterization of the Metric Completion of
Immersed Open-Curve Spaces}
\author{Ronny Gelman \\
\normalsize Advisor: Cy Maor}
\date{}

\begin{document}



\renewcommand{\baselinestretch}{1.5}
\maketitle

\input{abstract}
\input{acknowledgements}

\tableofcontents

\newpage


\input{introduction}

\input{Background}

\input{results}
\input{Discussion}

\bibliographystyle{alpha} 
\bibliography{thesis}

\end{document}

%% file: abstract.tex
\begin{abstract}
The completeness properties of spaces of immersed curves equipped with reparametrization-invariant Riemannian metrics have recently been the subject of active research. This thesis studies the metric completion of spaces of immersed open curves endowed with Sobolev-type metrics and examines a previously proposed conjecture that suggests the metric completion consists of a single additional point, representing all vanishing-length Cauchy sequences. We disprove the conjecture in the setting of real-valued immersed curves by demonstrating the existence of multiple distinct limit points. Furthermore, we provide a nearly complete characterization of the metric structure of the metric completion in this case. These results lead to a revised conjecture regarding the structure of the metric completion in more general settings.
\end{abstract}

%% file: acknowledgements.tex
\subsubsection*{Acknowledgements}
I would like to express my deepest gratitude to my advisor, Prof. Cy Maor, for his support and guidance.
Our many discussions were instrumental to the development of this thesis, and more importantly, my growth as a mathematician.
His mentorship extended well beyond the technical aspects of this work. I am immensely grateful for his invaluable advice regarding my future aspirations, help and encouragement in pursuing new opportunities, and his patient and thoughtful feedback, even when my ideas were scattered and disorganized. \par
I am also extremely thankful to my wife, Alex, for her emotional support through the many ups and downs during this time. She listened when I repeatedly shared fragments of ideas, encouraged me when I failed, and celebrated my happiness when slight progress was made. \par
I am further grateful to my friends, family, and colleagues from the Einstein Institute of Mathematics and from the Federmann Center for the Study of Rationality for their support.

%% file: introduction.tex
\section{Introduction}
\label{chap:intro}
In this thesis, we consider the space of Euclidean-valued immersed open curves, defined as 
\[ \mathcal{I}^n\big{(}[0,1],\mathbb{R}^d\big{)}\coloneqq \bigg\{c\in H^n\big((0,1),\mathbb{R}^d\big)\bigg| \  c'(\theta)\ne 0 \quad  \forall \theta \in (0,1) \bigg\} \]
for $\mathbb{N}\ni n\ge 2$  and $ d\in \mathbb{N} $.
It is an open subset of the Hilbert space $ H^n\big((0,1),\mathbb{R}^d\big) $, and as such, it can be viewed as a Hilbert manifold for which the tangent space at any point $ c $ is 
\[ T_c \mathcal{I}^n\big{(}[0,1],\mathbb{R}^d\big{)}\cong  H^n\big((0,1),\mathbb{R}^d\big). \]
We endow it with a Riemannian metric of the following type
\[  G_c(h,k)  \coloneqq \sum_{i=1} ^n a_i \intop_0 ^1 \langle \nabla_{\partial_s} ^i h,\nabla_{\partial_s} ^i k\rangle_{\mathbb{R}^d}\ ds,\quad a_0,a_n>0,\ a_i\ge0, \] 
where $ \nabla_{\partial_s}h=\frac{1}{|c'|}\partial_\theta h \ $  and $ ds=|c'|d\theta$.

This structure results in a metrically-incomplete space, as demonstrated by the following simple construction of a path leaving the space in finite time:
Let $\|v\|=1$ be a unit vector in $\mathbb{R}^d$, and define
\begin{align*}  
c:[0,1)&\rightarrow \mathcal{I}^n\big{(}[0,1],\mathbb{R}^d\big{)} \\
 t&\mapsto c(t)(\theta)\coloneqq v(1-t)\theta.
\end{align*} 
This path starts from the initial curve $c(0):[0,1]\rightarrow \mathbb{R}^d$, a constant-speed unit segment in direction $v$, and linearly shrinks it as $t\rightarrow1$. A direct calculation shows this path has finite length w.r.t.\ the Riemannian metric $G$. The limit curve, however, must be the constant curve at the origin, which is not an immersion and, therefore, is not in $\Ind$.

The purpose of this thesis is to explore the metric completion of the space $\mathcal{I}^n\big{(}[0,1],\mathbb{R}^d\big{)}$. It follows from \cite[Theorem~6.3]{Bauer2020} that any non-convergent Cauchy sequence $ (c_m)_{m=1}^\infty$  must have vanishing curve length, i.e.,
\[ \lim_{m\rightarrow\infty} \int_0 ^1 |  c_m ' | \  d\theta =0.\]
Moreover, \cite[Question~6.4]{Bauer2020} 
conjectures the following:
\begin{conj}
\label{conj: Thesis question}
    Let $\Ind$, for $d\ge1$ and $n\ge2$, be endowed with the Riemannian metric $G$, and let $d_G$ be the induced geodesic distance on $\Ind$.
    Then, the metric completion of $\left(\Ind,d_G\right)$ consists of $\Ind\cup\{0\}$, where ${0}$ is the limit point of all vanishing-length Cauchy sequences.
    \end{conj}

\Cref{subsec:example analysis} focuses on extending the existing example of a diverging Cauchy sequence, \cite[Example~6.1]{Bauer2020},  by examining straight-line segment curve sequences. We establish new estimates on the geodesic distance of certain types of paths, and use these to conclude that the limit point of a vanishing-length straight-line Cauchy sequence depends solely on the choice of parametrizations (\Cref{cor: rotation-translation invariance}). As a consequence, we ask the following:
Is it possible that two curve sequences of the form $\ell_m e_1\varphi(\theta)$, and $\ell_m e_1\psi(\theta)$, for $\psi\ne\varphi\in\diff$ and $\ell_m\searrow0$, converge to the same limit point in the metric completion space? This is a question that can also be examined in one-dimensional settings $\In$.\par

\Cref{subsec:one dimensional resluts}, answers the question above negatively for the case of $ d=1$ and a metric $G$ satisfying $a_2>0$, and, in particular, disproves \Cref{conj: Thesis question}. In these settings, we prove
\Cref{thm: distinct limit points}, which establishes the existence of a family of distinct limit points, each corresponding to a $H^2$-diffeomorphism of the interval $[0,1]$. Furthermore, \Cref{cor: final char} provides a nearly-complete characterization of the metric completion space.\par

\Cref{subsec: extension to multi-dimensional case} presents additional estimates on paths in $\Ind$, with $d>1$, and provides an additional construction of a diverging Cauchy sequence. \Cref{chap: Discussion} briefly touches on the possible future directions for extending the results in \Cref{subsec:one dimensional resluts} to more general settings.

%% file: Background.tex
\section{Background and Settings}
\label{chap: Background and settings}

This section is structured as follows: 
\Cref{subsec:Hilbert Manifold} presents the basic definitions of Hilbert manifolds, Riemannian metrics, and geodesic distance, highlighting the distinctions from finite-dimensional manifolds.\par
In \Cref{subsec: open curve immersions}, we introduce $\Ind$, the space of immersed open curves, and prove the basic properties required for the subsequent analysis. 
 \Cref{subsec: closed curve} introduces the space of immersed closed curves $\mathcal{I}^n\big(S^1,\mathbb{R}^d\big)$ as a sub-manifold of $\Ind$. \Cref{subsec: Previous Work} discusses previous work on the completeness properties of immersed curve spaces, and specifically the differences between the spaces of closed and open curve immersions. 

\subsection {Hilbert Manifolds: Basic Definitions}
\label{subsec:Hilbert Manifold}
We begin with a brief presentation of Hilbert manifolds, Riemannian metrics, and the induced geodesic distance.   \par

\begin{definition}
   A (real) \textbf{Hilbert space} is a real inner product space, which is complete with respect to the metric induced by the inner product.
\end{definition}
The definition of Hilbert manifolds follows the lines of the classical finite-dimensional case, with the distinction that the model space used in the definition of the charts is a Hilbert space rather than a finite-dimensional Euclidean space. This definition can be extended to locally convex topological vector spaces as model spaces. A detailed review of such spaces, which are outside the focus of this thesis, can be found in the first chapters of \cite{SchmedingAlexander2022AItI}.

\begin{definition}
    Let $\mathcal{H}$ be a Hilbert space, a \textbf{ smooth atlas}, $\mathcal{A}$,  for a Hausdorff topological space  $M$ is a collection of maps  $\varphi:U_\varphi\rightarrow V_\varphi$ where $V_\varphi\subset \mathcal{H}$ and $U_\varphi\subset M$ are open subsets, and $\varphi$ is a homeomorphism, satisfying
    \begin{enumerate}
        \item $M=\bigcup_{\varphi\in\mathcal{A}}U_\varphi$.
        \item For all $\varphi,\psi\in\mathcal{A}$, the composition $\varphi\circ\psi^{-1}$ is a smooth map $\mathcal{H}\supset V_\psi\rightarrow\mathcal{H}$.
    \end{enumerate}
    Each map $\varphi$ is called a \textbf{chart}.
    We say that a pair of smooth atlases $\mathcal{A},\mathcal{A}'$ for the space $M$, are equivalent if their union $\mathcal{A}\cup\mathcal{A}'$ is also a smooth atlas for $M$.
\end{definition}
The equivalence of smooth atlases is an equivalence relation, which allows the following definition:
\begin{definition}
    A \textbf{smooth Hilbert manifold}, $\big{(}M,\mathcal{A}\big{)}$, consists of a Hausdorff topological space $M$ with an equivalence class of smooth atlases equivalent to $\mathcal{A}$.
\end{definition}

\begin{remark*}
    Past this initial definition, we omit specifying the atlas of a smooth Hilbert manifold unless necessary for clarity. 
\end{remark*}

For a point $p\in M$, the tangent space to $M$ at $p$ is defined via $M$-valued $C^1$ curves passing through $p$.
\begin{definition}
\label{def: tangent space}
    We say two $C^1$-curves  $\gamma,\eta$ satisfying $\gamma(0)=p=\eta(0)$ are equivalent if for some chart, $\phi$, of M around $p$,
    \[(\phi\circ\gamma)'(0)=(\phi\circ\eta)'(0).\]
    An equivalence class $[\gamma]$ is called a \textbf{tangent vector} to $M$ at $p$. We define the \textbf{tangent space} to $M$ at $p$, $T_pM$, as the set of all such equivalence classes. \par
    At any point $p\in M$, the tangent space $T_pM$ has a natural isomorphism to the modeling space $\mathcal{H}\cong T_pM$ by the map taking any $y\in V_\phi\subset\mathcal{H}$ to the tangent vector defined as the equivalence class of the path $t\mapsto\phi^{-1}(p_\phi+ty)$, where $p_\phi\coloneqq\phi^{-1}(p)$ and $\phi$ is a chart around $p$.
\end{definition}

Unlike in the finite-dimensional case, defining Riemannian metrics on infinite-dimensional manifolds require more careful consideration:
\begin{definition}
    Let $M$ be a Hilbert manifold, a \textbf{weak Riemannian metric} $g$ on $M$ is a smooth choice of inner products  
    \[p\mapsto g_p, \quad g_p:T_pM\times T_pM\rightarrow \mathbb{R}.\]
    If for every $p\in M$, the topology induced on $T_pM$ by a weak Riemannian metric $g$ coincides with the topology of the modelling space under the identification described in \Cref{def: tangent space}, then $g$ is said to be a \textbf{strong Riemannian metric}.  \par
    The manifold $M$ will be called a \textbf{weak Riemannian manifold} or \textbf{strong Riemannian manifold}, respectively.
\end{definition}

Riemannian metrics allow for a way of measuring the length of curves, which, in turn, defines the geodesic distance between any pair of points $x,y\in M$:
\begin{definition}
\label{def: riem dist}
   For a weak Riemannian manifold $\big{(}M,g\big{)}$, the \textbf{length} of a  $C^1$-curve $\gamma:[a,b]\rightarrow M$, is defined as 
    \begin{equation}
        \len(\gamma)\coloneqq\intop_a^b\sqrt{g_{\gamma(t)}\big{(}\gamma'(t),\gamma'(t)\big{)}}dt.
    \end{equation}
    For a pair of points $x,y\in M$, we denote by $\Gamma(x,y)$ the set of all continuous piecewise $C^1$ curves joining $x$ and $y$, i.e.,
    \begin{equation}
        \label{eq: Gamma def}
        \Gamma(x,y)\coloneqq\big\{\gamma\in C^0\big([0,1],M\big) \ \big| \  \gamma(0)=x, \ \gamma(1)=y, \ \gamma \text{ is piecewise } C^1\big\}.
    \end{equation}
    The \textbf{geodesic distance} between $x$ and $y$ is defined as 
    \begin{equation}
        d_G(x,y)\coloneqq\inf_{\gamma\in\Gamma(x,y)}\len(\gamma).
    \end{equation}
\end{definition}

On weak Riemannian manifolds, the geodesic distance is only guaranteed to be a pseudo-distance, which means it may fail to separate points. On strong Riemannian manifolds, the geodesic distance is guaranteed to be a metric on each connected component \cite[Theorem~4.11]{SchmedingAlexander2022AItI}.

\subsection{The Space of Immersed Open Curves}
\label{subsec: open curve immersions}
In this section, we describe the construction of the space of Euclidean-valued immersed open curves $\Ind$, for $n\ge2$, and $d\ge1$. As mentioned in the introduction, the space is viewed as a Hilbert manifold, on which we define strong Riemannian metrics that induce a metric structure.
A key feature of the metrics we will define is their reparametrization invariance, which is important both for theoretical implications and practical uses in shape analysis. For example, as shown in \cite{Bauer_2014}, reparametrization invariant strong Riemannian metrics on $\Ind$ induce a non-degenerate metric on the corresponding shape space of unparametrized curves.\\

We begin by defining the
Hilbert space $ H^n\big((0,1),\mathbb{R}^d\big)$, on which $\Ind$ is modeled:
\begin{definition}
For a fixed Euclidean space $ \mathbb{R}^d $, the Sobolev space
\[\bigg(H^n\big((0,1),\mathbb{R}^d\big),\|\cdot\|_{H^n}\bigg),\quad \|f\|_{H^n}\coloneqq\bigg(\sum_{0\le i\le n}\|f^{(i)}\|^2_{L^2}\bigg)^\frac{1}{2}\]is a Banach space, which consists of all functions $ f:(0,1)\rightarrow \mathbb{R}^d $ that admit to weak derivatives of order $ 1\le k\le n$,  such that $ f $ as well as $ f^{(k)} $ for $ k=1,...,n$,  have finite $ L^2 $ norm.\par
It admits to a inner product given by 
\[ \langle f,g \rangle \coloneqq \sum_{k=1}^n \intop_0^1 \langle f^{(k)}(\theta),g^{(k)}(\theta)\rangle d\theta, \] 
turning it into a Hilbert space.
\end{definition}
This can be seen, for example, in  \cite[chapter~7]{leoni2009first}.
\begin{remark*}
For $f\in H^n\big((0,1),\mathbb{R}^d\big)$, we denote by $\|f\|_{\Dot{H}^n}$ the semi-norm: 
\[\|f\|_{\Dot{H}^n}\coloneqq\|f^{(n)}\|_{L^2}.\]
\end{remark*}
\begin{prop}\label{prop:abs cont} (\cite[chapter~7]{leoni2009first}) A function $ f \in H^n\big((0,1),\mathbb{R}^d\big) \ $ if and only if it has an absolutely continuous representative $ \overline{f} $ that is $ n-1 $ times continuously differential in the classical sense, with all $ n-1 $ derivatives being absolutely continuous, and that is almost-everywhere $ n $ times differentiable , such that 
\[\overline{f}^{(k)}\in L^2\big{(}(0,1),\mathbb{R}^d\big{)},\quad \forall k=0,...,n.\]
\end{prop} 

Each absolutely continuous curve $c:(0,1)\rightarrow\mathbb{R}^d$ can be uniquely extended to an absolutely continuous map $[0,1]\rightarrow\mathbb{R}^d$. By choosing absolutely continuous representatives, we may assume curves are defined on the closed interval $[0,1]$ rather than $(0,1)$.  Having defined the Hilbert space $ H^n\big((0,1),\mathbb{R}^d\big)$, we make the following observation:

\begin{prop}
\label{prop:openess}
For $n\ge 2$, the set 
\[ \Ind \coloneqq\{ c\in  H^n\big((0,1),\mathbb{R}^d\big) :c'(\theta)\ne0 \quad \forall \theta \in [0,1] \} \]
is open in  $  H^n\big((0,1),\mathbb{R}^d\big) $.
\end{prop}
\begin{proof}
Assume $ c \in  H^n\big((0,1),\mathbb{R}^d\big) $ such that $c'(\theta)\ne 0$ for all $\theta \in [0,1]$. In particular, by continuity of $c'$ over a compact domain, $ |c'(\theta)|
\ge \delta > 0$ for some $\delta>0$. Let also $d\in  H^n\big((0,1),\mathbb{R}^d\big) $, and notice that $c',d'\in H^1\big((0,1),\mathbb{R}^d\big)$.
Applying the standard Sobolev inequality  \cite[Theorem~7.34]{leoni2009first}
\[ \|u\|_{L^r(I)} \le \ell ^{1/r-1/q}\|u\|_{L^q(I)} + \ell ^{1-1/p-1/r} \|u'\|_{L^p(I)} \] where $\ell$ marks the length of the interval $I$,
to $u=c'-d'$ with $r=\infty$ and $p=q=2$ shows that 
\[ \|c'-d'\|_{L^\infty} \le C\|c'-d'\|_{H^1} \]
with $C$ being an universal constant independent of $c,d$. 
Consider the ball $B$ in $H^n\big((0,1),\mathbb{R}^d\big)$ of radius $\frac{\delta}{2C}$ around $c$. For any $d\in B$ we have
\[\|c'-d'\|_{H^1}\le\|c-d\|_{H^2}<\frac{\delta}{2C}\Rightarrow\|c'-d'\|_{L^\infty}<\frac{\delta}{2}.\] 
Since $|c'(\theta)|>\delta$ for all $\theta\in[0,1]$ and $d'$ is continuous, 
\[|d'(\theta)|>\frac{\delta}{2} \quad \forall\theta\in[0,1],\]
ensuring that $d$ is an immersion.
\end{proof}
As an open subset of the Hilbert space $H^n\big((0,1),\mathbb{R}^d\big)$, $\Ind$ has a natural manifold structure:
\begin{prop}
    $ \Ind $  with the atlas consisting of the global inclusion chart,
    \[\iota:\Ind\rightarrow H^n\big((0,1),\mathbb{R}^d\big), \]
    forms a $C^\infty$-Hilbert manifold. The 
    tangent space at a curve $c\in\Ind$ is the space of $H^n$-vector fields along the curve $c$:
    \[  T_c\Ind \cong H^n\big((0,1),\mathbb{R}^d\big),  \]
    with the identification given by the map $h\leftrightarrow[t\mapsto c+th]\in T_c\Ind$.
\end{prop}

A seemingly natural choice of strong Riemannian metrics on $\Ind$ can be defined  by taking the pointwise product
\[g_c(h,k)\equiv\langle h,k\rangle_{H^n} \]
for any pair of tangent curves $h,k\in T_c\Ind$. 
However, this choice of Riemannian metrics has a major drawback: it is not reparametrization invariant. That is to say, for some pairs of curves $c_1,c_2\in\Ind$, it is possible to choose a $H^n$-diffeomorphism of the interval $[0,1]$, $\varphi$, such that
\[\text{dist}_g(c_1,c_2)\ne \text{dist}_g(c_1\circ\varphi,c_2\circ\varphi)\]
where the distance $\text{dist}_g$ is the geodesic distance induced by $g$.
The need for reparametrization invariant metrics comes from a primary area of application of immersed curve spaces, Shape Analysis. Reparametrization invariant Riemannian metrics on $\Ind$ induce a metric on the quotient space 
\[\Ind\slash\diff,\]
the space of unparametrized curves, which is referred to as the shape space. In this case, the quotient map is a submersion of manifolds and allows lifting properties that simplify the study of geodesics in the shape space. Shape spaces and their properties under various Riemannian metrics are discussed more thoroughly in \cite[section~3]{michor2006riemanniangeometriesspacesplane}.\par
Due to this, we define another strong Riemannian metric on $\Ind$:
\begin{definition}
At any curve $c\in\Ind$ for $h,k\in T_c\Ind$  \par
define
    \begin{equation} 
    \label{eq: type of metric}
        G_c(h,k)\coloneqq\sum_{i=0}^n a_i \intop_0^1 \langle\nabla^i_{\partial_s}h, \nabla^i_{\partial_s}k\rangle ds,
    \end{equation}
where $a_i$ are non-negative constants for $ i=0,...,n$ with $a_0,a_n>0 $, the inner product $\langle\cdot,\cdot\rangle$ is the standard Euclidean inner product, 
\begin{equation}
\label{def:der_int}
\nabla_{\partial_s}h\coloneqq \frac{1}{|c'|}(\partial_\theta h), \quad ds\coloneqq|c'|d\theta 
\end{equation}
are differentiation and integration with respect to arc length, and $|\cdot|$ is used to indicate the standard Euclidean 2-norm.
\end{definition}

Notice that for $c\in \Ind$, since $c'$ is absolutely continuous, exists some $\delta>0$  such that 
\[|c'(\theta)|\ge\delta>0 \quad \forall \theta\in[0,1],\]
hence $\nabla_{\partial_s}$ is well defined, and for any $h\in H^1\big((0,1),\mathbb{R}^d\big)$, $\nabla_{\partial_s}h$ remains $L^2$ integrable, thus $G_c$ is well defined. \par
Notice also that in \eqref{def:der_int}, both  $\partial_\theta h$ and $c'$ indicate standard derivatives with respect to $\theta$. We will continue with this use of notation throughout to distinguish the derivative of the base curve $c$ from the differentiation of tangent vectors $h\in  H^n\big((0,1),\mathbb{R}^d\big)$ w.r.t.\ to $\theta$.\par

%
Next, we will demonstrate that the Riemannian metric $G$ is indeed reparametrization invariant under $H^n$-reparametrizations of $[0,1]$:
\begin{definition}
\label{def: diffeomorphism}
For $\ n\ge 2$, the set of $\ \mathbf{H^n}$\textbf{-diffeomorphisms} mapping $[0,1]$ onto itself will be denoted as
\begin{equation}
    \diff\coloneqq\bigg\{\varphi\in\mathcal{D}^1\big{(}[0,1]\big{)}\bigg|\varphi\in H^n([0,1]) \bigg\},
\end{equation}
where \[\mathcal{D}^1\big{(}[0,1]\big{)}\coloneqq\bigg\{d:[0,1]\rightarrow[0,1]\bigg| d\text{ is a $C^1$-homeomorphism of $[0,1]$ onto itself}\bigg\}.\]
 \par
We say that $\varphi\in\diff$ is \textbf{endpoint-fixing} if $\varphi(0)=0$ and $\varphi(1)=1$.  We denote,
\[\diffp\coloneqq\bigg\{\varphi\in\diff\bigg| \ \varphi \text{ is endpoint-fixing}\bigg\}.\] Similarly, we say  that $\varphi$ is \textbf{endpoint-switching} if $\varphi(0)=1$ and $\varphi(1)=0$, and we define $\diffm$ accordingly.
\end{definition}
\begin{prop}
\label{prop:diff is group}
$ \diff $ is a group when considered with composition $\varphi\psi=\varphi\circ\psi$ as the group action. Moreover, $  \diff$ acts from the right on $\Ind$ by reparametrization, i.e., by right composition $c.\varphi=c\circ\varphi$ for $\varphi\in\diff$ and $c\in\Ind$.
\end{prop}
\begin{proof}
    It is straightforward to verify that the composition of two diffeomorphisms and the inverse of a diffeomorphism are also diffeomorphisms. Hence, it remains to prove that they belong to $H^n\big((0,1),\mathbb{R}\big)$ for $n\ge2$. \par 
    The regularity requirement for composition is a direct consequence of \cite[Lemma~3.1]{ebin1970manifold}, which implies that if $f\in H^k\big((0,1),(0,1)\big)$ for $k\ge2$, satisfies $|f_\theta|>0$ and  $g\in H^k\big((0,1),\mathbb{R}\big)$, then $g\circ f\in H^k\big((0,1),\mathbb{R}\big)$. Regarding the regularity of the inverse, for any $\varphi\in\diff$ where $n\ge2$, by the inverse function theorem, we conclude that $\varphi^{-1}$ is $n-1$ times continuously differentiable. Thus, it remains to show that $\varphi^{-1}$ admits an $n$-th derivative and has finite $\Dot{H}^n$-norm. For $n=2$, by the chain rule,
    \[\partial_\theta\big(\varphi^{-1}\big) = \frac{1}{\varphi_\theta\circ\varphi^{-1}}
    \] 
    which is bounded with respect to the supremum norm on $[0,1]$. Differentiating again results in
    \[\partial_\theta^2(\varphi^{-1})=\frac{-\varphi_{\theta\theta}\circ\varphi^{-1}\cdot(\varphi^{-1})_\theta}{(\varphi_\theta\circ\varphi^{-1})^2}=\frac{-1}{(\varphi_\theta\circ\varphi^{-1})^3}\cdot\varphi_{\theta\theta}\circ\varphi^{-1}.\]
    Since $\frac{1}{(\varphi_\theta\circ\varphi^{-1})^3}$ is bounded on $[0,1]$ and $\varphi_{\theta\theta}$ is $L^2$ integrable, we conclude that 
    \begin{equation}
    \label{eq:Holder ieq.}
    \|\partial_\theta^2(\varphi^{-1})\|_{L^2}\le\|\frac{-1}{(\varphi_\theta\circ\varphi^{-1})^3}\|_{L^\infty}\cdot\|\varphi_{\theta\theta}\circ\varphi^{-1}\|_{L^2}<\infty,
    \end{equation}
    where we used the fact that $\varphi^{-1}\in \mathcal{D}^1\big([0,1]\big)$ to conclude  $\|\varphi_{\theta\theta}\circ\varphi^{-1}\|_{L^2}$ is finite by change of variable. \par 
   For $n>2$, we can continue with an inductive argument.
   \cite[Lemma~3.2] {ebin1970manifold} implies for $k\ge1$ and  $f,g\in H^k\big((0,1),\mathbb{R}\big)$, that $f\cdot g\in H^k\big((0,1),\mathbb{R}\big)$. Since $\partial^2_\theta(\varphi^{-1})=\frac{1}{(\varphi_\theta\circ\varphi^{-1})^3}\cdot\varphi_{\theta \theta}\circ\varphi^{-1}$, and $\partial_\theta\varphi^{-1}=\frac{1}{\varphi_\theta\circ\varphi^{-1}}\in H^{n-2}\big((0,1)\big)$ by assumption, it is enough to show that $\varphi_{\theta\theta}\circ\varphi^{-1}\in H^{n-2}\big((0,1),\mathbb{R}\big)$ as well. For $n>3$ this follows again by \cite[Lemma~3.1]{ebin1970manifold} since $n-2\ge2$, while for while for $n=3$ we have $\partial_\theta(\varphi_{\theta\theta}\circ\varphi^{-1})=\varphi^{(3)}\circ\varphi^{-1}\cdot\frac{1}{\varphi_\theta\circ\varphi^{-1}}$ and the claim follows similarly to \eqref{eq:Holder ieq.}. 
\end{proof}

It is not difficult to verify that the Riemannian metric $G$ is invariant under the action of $\diff$:
\begin{prop}
\label{para.inv}
    Let $\varphi\in\diff$, then at any $c\in \Ind$
    \[ G_{c\circ \varphi}(h\circ\varphi,k\circ\varphi)=G_c(h,k) \quad \forall h,k\in   H^n\big((0,1),\mathbb{R}^d\big) \]
\end{prop}
\begin{proof}
It is sufficient to show that for any $c\in \Ind $ and any $h,k\in  H^n\big((0,1),\mathbb{R}^d\big) $, the following equality holds for $i=0,1....,n$
\[\intop_0^1\langle \nabla_{\partial_{s(c)}}^i h,\nabla_{\partial_{s(c)}}^i k\rangle |c'|d\theta = \intop^1_0 \langle \nabla_{\partial_{s(c\circ\varphi)}}^i h\circ \varphi,\nabla_{\partial_{s(c\circ\varphi)}}^i k\circ \varphi \rangle|(c\circ\varphi)'|d\theta,\]
where $\nabla_{\partial_{s(c)}},\nabla_{\partial_{s(c\circ\varphi)}}$ are used to indicate the differentiation w.r.t.\ base curves $c,c\circ\varphi$ respectively. \par
Notice that the equality for $i=0$ holds by the change of variable formula. 
For higher derivatives, by the chain rule, for every $\theta\in[0,1]$ \[\nabla_{\partial_{s(c\circ\varphi)}}(h\circ\varphi)(\theta)=\frac{\partial_\theta (h\circ\varphi)}{|(c\circ\varphi)'|}(\theta)=\pm \frac{\partial_\theta h}{|c'|}\big(\varphi(\theta)\big)=\pm\big(\nabla_{\partial_{s(c)}}h\big)\big(\varphi(\theta)\big)\]
where the sign is determined by whether $\varphi$ is endpoint-fixing or endpoint-switching.
Applying this equality iteratively shows \[\nabla^i_{\partial_{s(c\circ\varphi)}}(h\circ\varphi)(\theta)=(\pm1)^i\nabla^i_{\partial_{s(c)}}h\big(\varphi(\theta)\big).\] Considering this, we have
\begin{equation*}
\begin{split}
&\intop^1_0\bigg{\langle} \nabla_{\partial_{s(c\circ\varphi)}}^i(h\circ\varphi),  \nabla_{\partial_{s(c\circ\varphi)}}^i (k\circ\varphi)\bigg{\rangle} |(c\circ\varphi)'|d\theta=\\
&\intop^1_0\bigg{\langle} (\pm1)^i\nabla^i_{\partial_{s(c)}}h(\varphi(\theta)),(\pm1)^i \nabla^i_{\partial_{s(c)}}k(\varphi(\theta))\bigg{\rangle} |(c\circ\varphi)'|d\theta=\\
&\intop^1_0\bigg{\langle}(\nabla^i_{\partial_{s(c)}}h)\circ\varphi(\theta), (\nabla^i_{\partial_{s(c)}}k)\circ\varphi(\theta)\bigg{\rangle} |c'\circ\varphi(\theta)||\varphi'(\theta)|d\theta=\\
&\intop^1_0\langle\nabla^i_{\partial_{s(c)}}h, \nabla^i_{\partial_{s(c)}}k\rangle|c'|d\theta,
\end{split}
\end{equation*}
where for the last equality we used the change of variable formula for $\varphi\in\mathcal{D}^1$.
\end{proof}

As in \ref{def: riem dist}, the Riemannian metric $G$ naturally induces a psuedo-metric on $\Ind $:
\begin{definition} 
The geodesic distance, \[ d_G:\Ind\times\Ind\rightarrow\mathbb{R}_{\ge0},\] is given by
\begin{equation}
\label{eq: geo dist}
    d_G(c_0,c_1)\coloneqq\inf_{\gamma\in\Gamma(c_0,c_1)} \len(\gamma)\coloneqq\inf_{\gamma\in\Gamma(c_0,c_1)}\intop_0^1 \sqrt{G_{\gamma(t)}\big(\partial_t\gamma(t),\partial_t\gamma(t)\big)}dt,
\end{equation}
where $\Gamma(c_0,c_1)$ is the set of piecewise-$C^1$ paths connecting $c_0$ and $c_1$, as defined in \eqref{eq: Gamma def}. 
\end{definition} 

\begin{prop} The Riemannian metric $G$ is a strong Riemannian metric, turning $\Ind$ into a strong Hilbert manifold. In particular, the geodesic distance \eqref{eq: geo dist} is point separating.
\end{prop}

The fact that $G$ is a strong Riemannian metric is easy to verify directly: by reparametrization invariance, we may assume $c$ is a constant speed curve. Thus, for any $h\in T_c\Ind$, $\|\nabla^k_{\partial_s}h\|^2_{L^2(ds)}=\frac{1}{\ell_c^{2k-1}}\|\partial_\theta^kh\|_{L^2(d\theta)}$, which implies
\[C^{-1} 
\|h\|_{G_c}\le \|h\|_{H^n}\le C\|h\|_{G_c},\]
where $C>0$ is a constant depending only on $c$ and $n$.

For $h\in T_c  \Ind$, we will often use $\|h\|_{L^2(ds)}$ to denote $\sqrt{\intop_0^1\langle h, h \rangle|c'|d\theta}$ and, similarly, $\|\nabla^i_{\partial_s}h\|_{L^2(ds)}\coloneqq\sqrt{\intop_0^1\langle\nabla^i_{\partial_s}h,\nabla^i_{\partial_s} h \rangle|c'|d\theta}$. With this notation, 
\[\len(\gamma)=\intop_0^1\sqrt{\sum_{0\le i\le n}a_i\|\nabla^i_{\partial_s}\gamma_t\|^2_{L^2(ds)}}dt,\]
where $\gamma_t\in  H^n\big((0,1),\mathbb{R}^d\big)$ is the derivative of the path $\gamma$ with respect to $t$. We will also often write $\gamma$ as $\gamma(t)(\theta)=\gamma(t,\theta)$ where $\theta$ is the parameter of the curve $\gamma(t)$.

\begin{cor}
\label{cor: dist inv}
The distance function is reparametrization invariant, meaning 
\[ d_G(c_0,c_1)=d_G(c_0\circ\varphi,c_1\circ\varphi) \quad \forall\varphi\in\diff.\]
\end{cor}
\begin{proof}
\Cref{prop:diff is group} guarantees, for any $\varphi\in\diff$, the existence of an inverse $\varphi^{-1}\in\diff$. Thus, the map taking a path $\gamma=\gamma(t,\theta)\in\Gamma(c_0,c_1)$ to the path \[\gamma.\varphi\coloneqq\gamma(t,\varphi(\theta))\in\Gamma(c_0\circ\varphi,c_1\circ\varphi)\]
is a bijection, as demonstrated by the existence of the inverse map 
\[\Gamma(c_0\circ\varphi,c_1\circ\varphi)\ni\eta\mapsto\eta.(\varphi^{-1})\in\Gamma(c_0,c_1).\] 
By \Cref{para.inv}, \[G_{\gamma.\varphi(t)}\left(\partial_t\gamma.\varphi,\partial_t\gamma.\varphi\right)=G_{\gamma(t)}\left(\partial_t\gamma,\partial_t\gamma\right)\qquad \forall t\in[0,1],\] thus,
 $\len(\varphi.\gamma)=\len(\gamma)$. We conclude 
\begin{equation*}
    \begin{split}
        d_G(c_0,c_1)=&\inf_{\gamma\in\Gamma(c_0,c_1)}\len(\gamma)=\inf_{\gamma\in\Gamma(c_0,c_1)}\len(\gamma.\varphi)\\ =&\inf_{\eta\in\Gamma(c_0\circ\varphi,c_1\circ\varphi)}\len(\eta)=d_G(c_0\circ\varphi,c_1\circ\varphi).
    \end{split}
\end{equation*}
\end{proof}

\begin{definition}
    Let $n\ge2$ and $d\ge1$. The \textbf{curve length} of $c\in\Ind$ is defined as 
    \[\ell_c\coloneqq\int_0^1|c'|d\theta=\|1\|_{L^2(ds)}^2.\]
\end{definition}
The following has been established by \cite[Lemma~6.6]{Bauer2020}:
\begin{prop}
\label{prop: Lip cont.}
    The map $c\mapsto\ell_c^{\nicefrac{3}{2}}$ is Lipschitz continuous on any metric ball in $\Big(\Ind,\ d_G\Big)$. Moreover, the Lipschitz constant in $B_r(c_0)$ depends only on $r,\ell_{c_0}$. 
\end{prop}
It follows that, for a given metric ball $B_r(c_0)\subset\Ind$,  \[|\ell_c^{\nicefrac{3}{2}}-\ell_{c_0}^{\nicefrac{3}{2}}|\le L(\ell_{c_0},r)r \qquad \forall c\in B_r(c_0).\] Thus, We conclude:
\begin{cor}
    \label{cor: boundness of ellc}
    The function $c\mapsto\ell_c$ is bounded on any metric ball in $\Ind$.
\end{cor}

\subsection{The Space of Immersed Closed Curves}
\label{subsec: closed curve}
Another space significant for the discussion in this thesis is the space of immersed closed curves, i.e., curves $S^1\rightarrow\mathbb{R}^d$.
Unlike $\Ind$, the space of closed curves is metrically complete as shown in \cite[theorem~4.3]{bruveris2015completenesspropertiessobolevmetrics}. Examining the differences between these two spaces provides a natural starting point for analyzing the metric completion of $\Ind$.
The space of immersed closed curves can be constructed following the same lines as the construction of the space of immersed open curves: \par
Define
\[\mathcal{I}^n\big(S^1,\mathbb{R}^d\big)\coloneqq\bigg\{c\in H^n\big(S^1,\mathbb{R}^d\big)\bigg|\ c'\ne0\quad\forall\theta\in S^1\bigg\},\] 
it is an open subset of the Hilbert manifold $H^n(S^1,\mathbb{R}^d)$ by the same argument as in  \Cref{prop:openess}. Hence, it can be viewed as a Hilbert manifold with the tangent space at a curve $c\in\mathcal{I}^n\big(S^1,\mathbb{R}^d\big)$
\[ T_c\mathcal{I}^n\big(S^1,\mathbb{R}^d\big) \cong H^n(S^1,\mathbb{R}^d). \]
We endow the manifold with similar Sobolev-type Riemannian metrics 
\begin{equation}
\label{eq: S1 metric}
g_c(h,k)\coloneqq\sum_{i=0}^n a_i \intop_{S^1} \langle\nabla^i_{\partial_s}h, \nabla^i_{\partial_s}k\rangle ds, \quad a_0,a_n>0,\ a_i\ge0.
\end{equation}
These, in turn, induce the metric
\begin{equation*}
\begin{split}
    d_g(c_0,c_1)\coloneqq&\inf_{\gamma\in\Gamma(c_0,c_1)} \len(\gamma)\coloneqq\inf_{\gamma\in\Gamma(c_0,c_1)}\intop_0^1 \sqrt{g_c(\gamma_t(t),\gamma_t(t))}dt\\
    =&\inf_{\gamma\in\Gamma(c_0,c_1)}\intop_0^1 \sqrt{\sum_{0\le i\le n}a_i\|\nabla^i_{\partial_s}\gamma_t\|^2_{L^2(ds)}}dt,
    \end{split}
\end{equation*}
where the infimum is taken over all paths in
\[\Gamma(c_0,c_1)=\bigg\{\gamma\in C^0\Big([0,1],\  \mathcal{I}^n\big(S^1,\mathbb{R}^d\big)\Big)\bigg|\gamma \text{ p.w. }C^1, \ \gamma(0)=c_0, \  \gamma(1)=c_1\bigg\}.\]

As in the case of open curves, the closed curve Riemannian metric is also reparametrization  
invariant; hence, the induced metric is as well. \par

Naturally, a closed curve $S^1\rightarrow\mathbb{R}^d$ can be thought of as a special case of an open curve $[0,1]\rightarrow\mathbb{R}^d$:
\begin{prop}
\label{prop:emb. clos-open}
Under the identification $S^1\cong\nicefrac{[0,2\pi]}{0\sim2\pi}$, there is a natural embedding 
\[\mathcal{I}^n\big(S^1,\mathbb{R}^d\big)\hookrightarrow\Ind \]
by taking a closed curve $c$ to an open curve $u$ defined as
\[ u(\theta)\coloneqq c(2\pi\theta), \quad \forall \theta \in [0,1]. \]
\end{prop}
It is important to emphasize that, in this context, having equal endpoints for an open curve is insufficient for it to be considered closed.
An open curve $c\in\Ind$ must satisfy
\[c^{(i)}(0)=c^{(i)}(1), \quad i=0,...,n-1 \]
to be considered closed, i.e. an element of $\mathcal{I}\big(S^1,\mathbb{R}^d\big).$ \par

This shows that $\mathcal{I}^n\big(S^1,\mathbb{R}^d\big)$ can be considered as an embedded sub-manifold of $\Ind$, thus allowing for two different ways of measuring the distance between a pair of closed curves. The first one, in terms of $d_g$ by restricting paths to those with the image contained in $\mathcal{I}^n\big(S^1,\mathbb{R}^d\big)$, and the second, in terms of $d_G$,  by allowing "ripping" and "stitching together" of closed curves when traveling between them.\par
If we chose the same set of coefficients for $G$ and $g$, the relation 
\[d_G\le d_g\]
is clear simply by the restriction of the set of paths considered for the definition of the distance functions.\par

 \cite[theorem~4.3]{bruveris2015completenesspropertiessobolevmetrics} establishes that $\mathcal{I}^n \big(S^1,\mathbb{R}^d\big)$ is a complete metric space with respect to $d_g$. This, however, does not guarantee it is complete as a sub-manifold of $\mathcal{I}^n\big(S^1,\mathbb{R}^d\big)$. For example, when considering the space of planar open curves, we do not know whether the sequence of circles with vanishing radius,
\begin{equation}
    \label{eq: vanishing circles}
(s_n)_{n=1}^\infty\subset\mathcal{I}^n([0,1],\mathbb{R}^2) \quad s_n\coloneqq\frac{1}{n}(cos(2\pi\theta),sin(2\pi\theta)),\end{equation}
 is Cauchy w.r.t.\ $d_G$, even though it is surely not Cauchy w.r.t.\ $d_g$ since it would mean that it converges to a curve of zero-length by continuity of $\ell_c$.

\subsection{Previous Work}
\label{subsec: Previous Work}
The completeness properties of immersed curve spaces with reparameterization invariant metrics have been studied on various occasions. There are three notions of completeness for a Riemannian manifold:
\begin{enumerate}[(a)]
    \item \label{num: met comp} Metric completeness with respect to the geodesic distance;
    \item \label{num: geo comp} Geodesic completeness: existence of geodesics for infinite time;
    \item \label{num: join comp} Geodesic convexity: the existence of a minimizing geodesic between any pair of distinct points.
\end{enumerate}
For finite-dimensional manifolds, the Hopf-Rinow theorem implies that \ref{num: met comp} and \ref{num: geo comp} are equivalent, and both imply \ref{num: join comp}. For infinite-dimensional manifolds, however, the Hopf-Rinow theorem does not hold \cite{Hopf-rinowAtkin}. In this case, it is only guaranteed, for strong Riemannian metrics, that \ref{num: met comp} implies \ref{num: geo comp}; any of the other entailments may fail.\par

It is shown in \cite{bruveris2014geodesic}, that $\mathcal{I}^n
\big(S^1,\mathbb{R}^2\big)$, endowed with Riemannian metrics of the form \eqref{eq: S1 metric} for $n\ge2$, is geodesically complete. Later, \cite{bruveris2015completenesspropertiessobolevmetrics}, presents a more extensive result, stating that $\mathcal{I}^n\big(S^1,\mathbb{R}^d\big)$ for any $d\ge1$ and $n\ge2$, with metrics as in \eqref{eq: S1 metric} satisfies \ref{num: met comp}, \ref{num: geo comp}, and \ref{num: join comp}. \par

Another family of immersion spaces that have been studied are spaces of immersed curves endowed with length-weighted Sobolev-type Riemannian metrics, i.e., metrics of the form \eqref{eq: S1 metric} with $a_i=a_i(\ell_c)$ being length dependent. An important family of such metrics is scale-invariant Riemannian metrics defined on $\mathcal{I}^n\big(D,\mathbb{R}^d\big)$ by 
\[g^{SI}_c(h,k)\coloneqq\sum_{i=0}^n A_i\cdot\ell_c^{2k-3}{} \intop_{D} \langle\nabla^i_{\partial_s}h, \nabla^i_{\partial_s}k\rangle ds, \quad A_0,A_n>0,\ A_i\ge0,\]
where either $D=S^1$ or $D=[0,1]$. In this case it is easy to verify that $g_c^{SI}(h,k)=g_{\lambda c}^{SI}(\lambda h,\lambda k)$ and thus $d_{g^{SI}}(c_0,c_1)=d_{g^{SI}}(\lambda c_0, \lambda c_1)$. \par

It is shown in \cite[Theorem~1.1]{bruveris2017completeness} that $\mathcal{I}^n\big(S^1,\mathbb{R}^d\big)$, for $n\ge2$, endowed with length-weighted Sobolev-type metrics, satisfies the three completeness properties under the condition that
\begin{equation}
    \label{eq: PW int conditions}
\underset{0\le k \le n}{\max} \int_0^1\ell^{\frac{1}{2}-k}\sqrt{a_k(\ell)}d\ell=\infty\quad \text{and} \quad\underset{0\le k \le n}{\max} \int_1^\infty\ell^{\frac{1}{2}-k}\sqrt{a_k(\ell)}d\ell=\infty.
\end{equation}
These conditions emerge as a way to ensure that any path in $\mathcal{I}^n\big(S^1,\mathbb{R}^d\big)$ with vanishing curve length or unbounded curve length has infinite length w.r.t.\ the induced geodesic distance. Notice that for $\mathcal{I}^n\big(S^1,\mathbb{R}^d\big)$, with $n\ge2$, both constant coefficient metrics as in \eqref{eq: S1 metric}, and scale invariant metrics as $g_c^{SI}$ satisfy \eqref{eq: PW int conditions}. \par 

This sequence of results was extended by \cite[Theorem 5.3]{Bauer2020} to spaces of immersed curves taking values in a finite-dimensional manifold $\mathcal{N}$. It implies, for $n\ge2$, that $\mathcal{I}^n\big(S^1,\mathcal{N}\big)$ with constant coefficient metrics as in \eqref{eq: S1 metric} satisfies \ref{num: met comp}, \ref{num: geo comp}, and \ref{num: join comp}. Moreover, it states that $\mathcal{I}^n\big(D,\mathcal{N}\big)$, with $n\ge2$ and $D=S^1$ or $D=[0,1]$, endowed with length-weighted Sobolev-type metrics satisfying for some $\alpha>0$ that either $a_1(\ell)>\alpha\ell^{-1}$ or both $a_0(\ell)>\alpha\ell^{-3}$ and $a_k(\ell)\ge\alpha\ell^{2k-3}$ for some $k>1$, also satisfies the three completeness properties.  \par

All of the immersion spaces considered above share a common property that plays a key role in the completeness argument: the map $c\mapsto\ell_c$ is bounded from above and away from zero on any metric ball. This property is significant when considering the Sobolev-type inequalities for $k=0,...,n-1$:
\begin{equation}
    \begin{split}
        \label{eq: Sobolev estimates}
        \ell_c^{2k}&\|\nabla^k_{\partial_s}h\|^2_{L^2(ds)}\lesssim_n\|h\|^2_{L^2(ds)}+\ell_c^{2n}\|\nabla^n_{\partial_s}h\|_{L^2(ds)}^2 \\
        \ell_c^{2k}&\|\nabla^k_{\partial_s}h\|^2_{L^\infty}\lesssim_n \ell_c^{-1}\|h\|^2_{L^2(ds)}+\ell_c^{2n-1}\|\nabla^n_{\partial_s}h\|_{L^2(ds)}^2.
    \end{split}
\end{equation}
The boundedness of $c\mapsto\ell_c$ and $c\mapsto\ell_c^{-1}$ is used to show that the geodesic distance induced by reparametrization invariant metrics is equivalent to that induced by standard $H^n$ metrics on every metric ball, as done in the proof of \cite[Theorem~5.3]{Bauer2020}. The fact that the metrics are equivalent on metric balls is what implies completeness in the cases above. \par

For an immersion space $\mathcal{I}^n\left(D,\mathbb{R}^d\right)$ where $n\ge2$, endowed with a Riemannian metric $\mathcal{G}$, the boundedness of $c\mapsto\ell_c$ on metric balls is satisfied whenever the $\mathcal{G}$ satisfies for some uniform  constant $C>0$ 
\[\|\nabla_{\partial_s}h\|_{L^2(ds)}\le C\max(1,\ell_c^{-1})\|h\|_{\mathcal{G}_c} ,\qquad \forall h\in T_c\mathcal{I}^n\big(D,\mathbb{R}^d\big).\]
This is established by \cite[Lemma~6.6]{Bauer2020} for the case of $\Ind$, and holds for  $\mathcal{I}^n\left(S^1,\mathbb{R}^d\right)$ since it is embedded in $\Ind$. Particularly, \eqref{eq: Sobolev estimates} implies $c\mapsto\ell_c$ is bounded on metric balls in $\Ind$ endowed with a constant coefficient Riemannian metric of type \eqref{eq: type of metric}.
The boundedness of $c\mapsto\ell_c^{-1}$ on metric balls, however, is a property harder to satisfy; and does not hold for $\Ind$ endowed with constant coefficient metrics, as \cite[Example~6.1]{Bauer2020} demonstrates. A sufficient condition for the boundedness of $c\mapsto\ell_c^{-1}$ on metric balls is provided in \cite[Lemma~5.9]{Bauer2020}. It implies, for $\mathcal{I}^n\big(D,\mathbb{R}^d\big)$ endowed with a Riemannian metric $\mathcal{G}$, where $n\ge2$ and $D=S^1$ or $D=[0,1]$, that if for any metric ball $B_r(c_0)\subset\mathcal{I}^n\big(D,\mathbb{R}^d\big)$ there exists some $C>0$  such that  
\[C\ell_c^{\nicefrac{-1}{2}}\|\nabla_{\partial_s}h\|_{L^2(ds)}\le\|h\|_{\mathcal{G}_c},\qquad \forall h\in T_c\mathcal{I}^n\big(D,\mathbb{R}^d\big),\ \forall c\in B_r(c_0),\]
then $c\mapsto\ell_c^{-1}$ is bounded on any metric ball. \par

In addition, the upper inequality of \eqref{eq: Sobolev estimates} shows that in case $c\mapsto\ell_c$ and $c\mapsto\ell_c^{-1}$ are both bounded on metric balls, higher-order constant coefficient Sobolev-type Riemannian metrics with $a_0,a_n>0$ dominate lower-order terms. This is the case, for example, in $\mathcal{I}^n
\big(S^1,\mathbb{R}^d\big)$ with metrics of the form \eqref{eq: S1 metric}. However, this is not the case in $\Ind$ with constant coefficient metrics as in \eqref{eq: type of metric}. As a result, the choice of $a_i=0$ for $0<i<n$ may influence the metric structure of $\Big(\Ind,d_G\Big)$ for $G$ with constant coefficients. For this reason, in the results we present in \Cref{subsec:one dimensional resluts}, we assume explicitly that $a_2>0$, even if $n>2$. \par

The main result concerning the completeness properties of $\Ind$ with constant coefficient metrics is \cite[Theorem~6.3]{Bauer2020}. It implies that any Cauchy sequence with non-vanishing curve length is convergent. The proof relies on the Lipschitz continuity of $c\mapsto\ell_c^{\nicefrac{3}{2}}$ on metric balls mentioned in \Cref{prop: Lip cont.}, it guarantees the existence of a neighborhood containing the sequence on which $\ell_c$ is both bounded and bounded away from zero. In this neighborhood, the techniques used to prove the completeness of the closed curve space can be employed. On the other hand, \cite[Example~6.1]{Bauer2020} constructs a path $(0,1]\rightarrow\mathcal{I}^n\big([0,1],\mathbb{R}^d\big)$ that illustrates the metric incompleteness of $\Ind$ with constant coefficient metrics. This example provides the starting point for further analysis of the metric completion, and is examined in \Cref{subsec:example analysis} below. \par
 
 The path constructed in \cite[Example~6.1]{Bauer2020} is a path of straight-line segment curves. It takes advantage of straight lines having identically zero curvature, which eliminates the effects of terms of order $n\ge2$ on the length of the path. We begin \Cref{chap:results} with a similar approach, focusing on straight-line segment curves as a starting point, and expand the existing example of diverging Cauchy sequences.

%% file: results.tex
\section{Results}
\label{chap:results}

\quad This section presents the thesis's contribution to the study of the metric completion of the space of immersed open curves. Throughout this section, we consider $\left(\Ind,d_G\right)$, where $G$ is of type \eqref{eq: type of metric}.\par 
\Cref{subsec:example analysis} examines three types of paths: linear shrinking, affine translation, and rotation. We provide bounds on the lengths of these paths and show that uniformly parametrized vanishing-length straight-line curve sequences are divergent Cauchy sequences.
\Cref{subsec:one dimensional resluts} examines whether it is possible that two such sequences with different parametrizations share the same limit point in the metric completion space. We examine this question in one-dimensional settings and prove, in this case, that the answer is negative and that multiple distinct limit points exist in the metric completion space. Moreover, we further characterize the completion space by bounding the set of possible limit points.
Lastly, \Cref{subsec: extension to multi-dimensional case} contains additional estimates on paths in $\Ind$ with $d>1$, and an additional construction of a path leaving the space in finite time.  \par
Throughout this section, we will reserve the use of the word `curve' for elements $c\in\Ind$, and `path' for paths $\gamma:I\rightarrow \Ind$ where $I\subset\mathbb{R}$ is an interval. Essentially, a path $\gamma$ is a `curve of curves' assigning to $t\in I$ the  curve $\gamma(t)$ in $\Ind$. 
When discussing a curve $c\in\Ind$, `starting point' will refer to $c(0)$. For a path $\gamma$ defined on a half-open interval, either $(0,1]$ or $[0,1)$, the `starting point' of $\gamma$  or `initial curve' is used to describe either $\gamma(1)$ or $\gamma(0)$ respectively. When discussing paths defined on closed intervals, we may refer to either of the endpoint curves as the initial curve. \par

\subsection{Analyzing the Existing Example of Diverging Cauchy Sequences}
\label{subsec:example analysis}

We begin by examining the basic example from \cite[Example~6.1]{Bauer2020}. Define the following path
\begin{equation}
\label{def: example path}
    \begin{split}
    \gamma:(0,1]&\rightarrow \mathcal{I}^n([0,1],\mathbb{R}^2)\\
    \gamma(t,\theta)&= \big{(}t\theta+f(t),g(t)\big{)},   
    \end{split}
\end{equation}
where $f,g:(0,1]:\rightarrow\mathbb{R}^2$ are continuous and piecewise differentiable.
The initial curve $\gamma(1)$ is a unit-speed segment of length 1 parallel to the $x$-axis with the starting of point $\big{(}f(1),g(1)\big{)}$. As $t\rightarrow0$, the length of the segment $\gamma(t)$ linearly decreases to zero. Simultaneously, $f$ and $g$, which are used as affine translations of the segment, position the curve such that the starting point, $\gamma(t,0)$, of the curve $\gamma(t)$ is $\big{(}f(t),g(t)\big{)}$.\par
As \cite[chapter~6]{Bauer2020} shows, the path in \eqref{def: example path} is of finite length under the assumption that 
\[\int_0^1|f'(t)|t^{1/2}dt<\infty\quad \text{and} \quad \int_0^1|g'(t)|t^{1/2}dt<\infty.\]
As a result, under these conditions, for any sequence $(0,1]\supset(t_m)_{m=1}^\infty, \ \text{s.t.} \  t_m
\longrightarrow0$, the corresponding sequence of curves, $\big{(}\gamma(t_m)\big{)}_{m=1}^\infty$, is a non-converging Cauchy sequence. 
Indeed, it is Cauchy since
\[d_G\big{(}\gamma(t_k),\gamma(t_m)\big{)}\le\len(\gamma|_{(0,\max{(t_k,t_m)}]})\underset{\max{(t_k,t_m)}\rightarrow0}{\longrightarrow}0,\]
where the limit follows from the above-mentioned claim that $\len(\gamma)<\infty$.
Yet, the only possible limit point would be a constant curve with an identically zero derivative, which does not belong to $\mathcal{I}^n\big([0,1],\mathbb{R}^2\big)$.\par
The integral constraints on $f,g$ are necessary to ensure that the path $\gamma$ in \eqref{def: example path} is of finite length. 
However, we will show below (\Cref{claim: straight-line Cauchy}) that any sequence of the form $\big{(}\gamma(t_m)\big{)}_{m=1}^\infty$  with $t_m\rightarrow 0$ is Cauchy, regardless of whether these bounds hold.   \par
\subsubsection{Bounds on Shrinking, Affine Translation, and Rotation}

Take any curve $c\in\Ind$; we examine the length of the path linearly shrinking $c$ up to zero length. 
\begin{lemma}
\label{lemma:shrink cost}
For any curve $c\in\Ind$, the length of the path 
\begin{equation*}
    \begin{split}
        \gamma:(0,1]&\rightarrow\Ind\\
        t&\mapsto tc
    \end{split}
\end{equation*}
is finite if and only if $c$ is a straight line, and in this case, there exists a constant $C=C(a_0,a_1)$, independent of $c$, such that 
\begin{equation}
    \len(\gamma)\le C\max(\ell_c^{1.5},\ell_c^{0.5})
\end{equation}
where $\ell_c$ is the length of the curve $c$.
\end{lemma}

\begin{proof}
As we saw in the proof of \Cref{cor: dist inv}, the length of a path is invariant to the action of $\diff$, i.e. \[\len(\gamma)=\len(\gamma.\varphi)=\len\big(\gamma(t,\varphi(\theta))\big) \quad \forall\varphi\in\diff.\]
Thus, without loss of generality, we can assume $c$ is constant-speed.
In this case, $|\gamma'(t,\theta)|=|t\gamma'|\equiv t\ell_c $  is independent of $\theta$.
Since $\gamma_t\equiv c$ at any time $t$, for any $i=1,...,n$, we have
\begin{equation*}
    \begin{split}
       \nabla^i_{\partial_s}\gamma_t(\theta)&=\nabla^i_{\partial_s}c(\theta)=\frac{1}{t^i\ell_c^i}\partial^i_\theta c(\theta) \quad \forall\theta\in[0,1],\\
    \|\nabla^i_{\partial_s}\gamma_t\|_{L^2(ds)}^2&=\intop_0^1\langle\frac{\partial^i_\theta c}{t^i\ell_c^i},\frac{\partial^i_\theta c}{t^i\ell_c^i}\rangle|t\ell_c|d\theta=\frac{1}{t^{2i-1}}\frac{\|\partial^i_\theta c\|^2_{L^2(d\theta)}}{\ell_c^{2i-1}}.
    \end{split}
\end{equation*}
Hence, 
\[\len(\gamma)=\intop_0^1\sqrt{\sum_{0\le i\le n}a_i\|\nabla^i_{\partial_s}\gamma_t\|_{L^2(ds)}^2}dt=\intop_0^1\sqrt{\sum_{0\le i\le n} \frac{a_i\|\partial^i_\theta c\|^2_{L^2(d\theta)}}{\ell_c^{2i-1}}\frac{1}{t^{2i-1}}}dt.\]
Since the integral $\int_0^1\frac{1}{t^\alpha}dt$ diverges if and only if $\alpha\ge1$,  $\len(\gamma)$ is finite if and only if  $ \|\partial^i_\theta c\|^2_{L^2(d\theta)}=0$ for every $i\ge2$ such that $a_i\ne0$, and particularly for $i=n$.

If so, $\partial^n_\theta c\equiv0$ and, by integration, every coordinate of $\partial_\theta c=c'$ is a polynomial of degree at most $n-1$. This would imply 
\[|c'|^2=\sum_{j\le d}(c'_j)^2\] 
is itself a polynomial, and since $|c'|^2(\theta)\equiv\ell_c^2$, it must be constant.  We write every coordinate as a polynomial, i.e., $\partial_\theta c_j=\sum_{i\le n-1}c_j^i\theta^i,\ c_j^i\in\mathbb{R}$. In order to show each $\partial_\theta c_j$ is constant, we consider the term of degree $2(n-1)$ of $|c'|^2$. We have,
\[0=\sum_{j\le d}(c_j^{n-1}\theta^{n-1})^2,\]
which implies $c_j^{n-1}=0$ for all $j\le d$. Repeating the argument for terms of order $0<i<n-1$ ensures that $\partial_\theta c_j \equiv c_j^0$ for all $j\le d$. By integration, we conclude $c$ is a straight line.\par
On the other hand, if $c(\theta)\coloneqq v\theta$ is a straight line, any derivative of order higher than one vanishes, and we are left with 
\[\len(\gamma)=\intop_0^1\sqrt{a_0\|v\theta\|_{L^2(ds)}^2+a_1\|\frac{v}{t|v|}\|^2_{L^2(ds)}}dt.\]
Explicitly,  
\[\len(\gamma)=\intop_0^1\sqrt{a_0\frac{|v|^3t}{3}+a_1\frac{|v|}{t} }dt\le\intop_0^1\sqrt{a_0\frac{t|v|^3}{3}}+\sqrt{a_1\frac{|v|}{t}} dt,\]
and since $\ell_c=|v|$, we have
\[\len(\gamma)\le C\max(\ell_c^{1.5},\ell_c^{0.5}) \]
where the constant can be taken to be $C=2(\sqrt{a_0}+\sqrt{a_1})$.
\end{proof}
As a result, we can deduce the following
\begin{cor}
\label{cor: shrink dist}
    For any straight-line curve $c\in\Ind$, and any $0<\lambda<1$:
    \[d_G\big{(}\lambda c, c\big{)}<C\max(\ell_c^{1.5},\ell_c^{0.5})\]
    where $\ell_c$ is the length of the curve $c$, and $C$ is the constant from \Cref{lemma:shrink cost}, depending only on the coefficients of the metric $G$.
\end{cor}
\begin{proof}
    Define a path $\gamma$ as in \Cref{lemma:shrink cost}, and consider 
    \[\Tilde{\gamma}\coloneqq\gamma\big|_{[\lambda,1]}.\]
    By definition, $d_G\big(\lambda c,c\big)\le\len(\Tilde{\gamma})$, and by \Cref{lemma:shrink cost} we conclude that
    \[d_G(\lambda c,c)\le\len(\Tilde{\gamma})\le\len(\gamma)\le C\max(\ell_c^{1.5},\ell_c^{0.5}).\]
    \end{proof}
Next, we consider the distance between a curve and an affine translation of it:
\begin{lemma}
\label{lemma: translation cost}
For any curve $c\in\Ind$ and any vector $v_0\in\mathbb{R}^d$,
\begin{equation}
    d_G(c,c+v_0)\le |v_0|\cdot\sqrt{a_0\ell_c},
\end{equation}
where the curve $c+v_0$ is defined as $\big{(}c+v_0\big{)}(\theta)=c(\theta)+v_0$, and $a_0$ is the coefficient of the degree-zero term of the metric $G$.
\end{lemma}
\begin{proof}
Define the linear interpolation path
\begin{equation*}
    \begin{split}
        \gamma:[0,1]&\rightarrow \Ind\\
        t&\mapsto c+tv_0.
    \end{split}
\end{equation*}
Notice $\gamma_t\equiv v_0$ is independent of $\theta$, and so for any $i>0$ we have $\nabla^i_{\partial_s}\gamma_t\equiv0$. Therefore, the length of $\gamma$ simplifies to 
\[\intop_0^1\sqrt{a_0\|\gamma_t\|^2_{L^2(ds)}}dt=\intop_0^1
\sqrt{a_0\int_0^1|v_0|^2|c'|d\theta}dt=|v_0|\sqrt{a_0\ell_c}.\]
\end{proof}

Lastly, we provide a bound on the distance between a curve $c$ and $Ac$, a rotation of it by some $A\in SO(n)$:
\begin{lemma}
\label{lemma: rot cost} There exists a constant $C$, dependent only on the dimension $d$, such that for any $c\in\Ind$ and any rotation matrix $A\in SO(d)$, 
\begin{equation}
    d_G(c,Ac)\le C\sqrt{G_c(c,c)},
\end{equation} 
where $\sqrt{G_c(c,c)}$ is the norm of $c\in T_c \Ind$ as a tangent vector at $c\in\Ind$.
\end{lemma}

\begin{proof}
We begin by noting that $SO(d)$ is a connected compact sub-manifold of $\mathbb{R}^{d^2}\cong M_{d\times d}(\mathbb{R})$, which we consider with the intrinsic geodesic distance $d_{SO(d)}$ induced by the standard Riemannian metric. As a result, $\text{diam}\big{(}SO(d)\big{)}<\infty$ with respect to the induced Riemannian metric. By definition, $SO(d)$ contains the identity element $I_d$, hence we can choose a path $\mathcal{A}:[0,1]\rightarrow SO(d)$ such that $\mathcal{A}(0)=I_d, \mathcal{A}(1)=A\in SO(d)$ and $\len(\mathcal{A})=d_{SO(d)}\big(I_d,A\big)\le\text{diam}\big{(}SO(d)\big{)}$. \par
We define a path $\gamma:[0,1]\rightarrow\Ind$ by 
\[\gamma(t,\theta)=\mathcal{A}(t)\big(c(\theta)\big).\]
Notice $\gamma_t$ is simply $\mathcal{A}_t\big(c(\theta)\big)$, and since $\mathcal{A}_t$ is linear and independent of $\theta$, 
\[\partial_\theta\gamma_t=\partial_\theta\mathcal{A}_t\big(c(\theta)\big)=\mathcal{A}_t\big(\partial_\theta c(\theta)\big).\]
Since $\mathcal{A}(t)$ is a linear orthogonal map for every $t$, 
\[|\partial_\theta\gamma(t,\theta)|=|\mathcal{A}(t)\big{(}c'(\theta)\big{)}|=|c'(\theta)|.\]
Thus, we conclude 
\begin{equation}
\label{eq: rot equiv der}
\nabla_{\partial_s}\gamma_t(\theta)=\nabla_{\partial_s}\mathcal{A}_t\big(c(\theta)\big)=\mathcal{A}_t\big(\nabla_{\partial_s}c(\theta)\big) \quad \forall \theta\in[0,1],
\end{equation}
where $\nabla_{\partial_s}c$ in the rightmost expression is taken at $T_c\Ind$, as opposed to $\nabla_{\partial_s}\gamma_t$ taken at $T_{\mathcal{A}(t)c}\Ind$.\par
Using the same argument, we can conclude similarly to \eqref{eq: rot equiv der}, that $\nabla_{\partial_s}^i\gamma_t(\theta)=\mathcal{A}_t\big(\nabla_{\partial_s}^ic(\theta)\big)$ for every $0\le i\le n$. 
Thus, by the orthogonality of $\mathcal{A}(t)$, for any $0\le i\le n$ we have:
\begin{equation*}
    \begin{split}
    \|\nabla^i_{\partial_s}\gamma_t\|^2_{L^2(ds)}=\intop_0^1|\mathcal{A}_t(\nabla^i_{\partial_s}c)|^2|\mathcal{A}(t)(c')|d\theta\le \intop_0^1\|\mathcal{A}_t\|_{op}^2|\nabla^i_{\partial_s}c|^2|c'|d\theta\\
    =\|\mathcal{A}_t\|_{op}^2\intop_0^1|\nabla^i_{\partial_s}c|^2 |c'|d\theta=\|\mathcal{A}_t\|_{op}^2\|\nabla^i_{\partial_s}c\|^2_{L^2(ds),}
\end{split}
\end{equation*}
where $\|\nabla^i_{\partial_s}c\|^2_{L^2(ds)}$ is taken at the tangent space $T_c\Ind$.\par
By the equivalence of norms on $\mathbb{R}^{d^2}$, $\|\mathcal{A}_t\|^2_{op}\le C' \|\mathcal{A}_t\|_{SO(d)}^2$,  for some constant $C'$ dependent only on $d$.
Therefore, when considering the length of $\gamma$, we find 
\begin{equation*}
    \begin{split}
        \len(\gamma)=\intop_0^1\sqrt{\sum_{0\le i \le n}a_i\|\nabla^i_{\partial_s}\gamma_t\|^2_{L^2(ds)} }dt\le\intop_0^1\sqrt{\sum_{0\le i \le n}a_i\|\mathcal{A}_t\|_{op}^2\|\nabla^i_{\partial_s}c\|^2_{L^2(ds)}}dt \\
        =\sqrt{\sum_{0\le i \le n}a_i\|\nabla^i_{\partial_s}c\|^2_{L^2(ds)}}\intop_0^1\sqrt{\|\mathcal{A}_t\|^2_{op}}dt=\sqrt{G_c(c,c)}\intop_0^1\sqrt{\|\mathcal{A}_t\|^2_{op}}dt\\
        \le\sqrt{G_c(c,c)}\cdot C'\intop_0^1\sqrt{\|\mathcal{A}_t\|_{SO(d)}^2}dt.
    \end{split}
\end{equation*}
The integral $\intop_0^1\sqrt{\|\mathcal{A}_t\|^2_{SO(d)}}dt$ is precisely  the length of the path $\mathcal{A}$ in $SO(d)$, therefore 
\[\len(\gamma)\le C'\text{diam}\big{(}SO(d)\big{)}\sqrt{G_c(c,c)}\coloneqq C\sqrt{G_c(c,c)}.\]
\end{proof}

\subsubsection{The Implication for Vanishing-length Straight-line Sequences}

The combination of the last three lemmas allows us to conclude the following: 
\begin{theorem}
\label{claim: straight-line Cauchy}
Let $n\ge2$, and let $(s_m)_{m=1}^\infty\subset\Ind$ be a sequence of uniformly-parametrized vanishing-length straight-line segment curves.
If $d>1$, $(s_m)_{m=1}^\infty$ is a divergent Cauchy sequence. If $d=1$, $(s_m)_{m=1}^\infty$ is a divergent Cauchy sequence if it lies within a single connected component of $\In$.
\end{theorem}
\begin{remark*}
    Any straight-line segment curve $s\in\Ind$ can be written uniquely as 
    \[s(\theta)=v+r\ell\varphi(\theta), \quad v\in\mathbb{R}^d, r\in S^{d-1}, \ell\in\mathbb{R}_+, \varphi\in\diffp.\]
    For $d=1$, under this identification, the two connected components of $\In$ are determined by $\text{sgn}(r)$ where $r\in S^0=\{-1,1\}$, as discussed in \Cref{subsec:one dimensional resluts} below.
    
\end{remark*}

\begin{proof}
Using the identification above, we write
\[ s_m(\theta)=v_m+r_m\ell_m\varphi(\theta),\] 
where $\ell_m\searrow0$ is the length, $r_m\in S^{d-1}$ is the direction, and $v_m\in\mathbb{R}^d$ is the starting point of the straight-line segment curve $s_m$. By \Cref{prop: Lip cont.}, $c\mapsto\ell_c$ is continuous, by which we conclude that $(s_m)_{m=1}^\infty$ is non-converging since $\ell_m\rightarrow0$ yet any curve $c\in\Ind$ has $\ell_c>0$. To prove it is Cauchy, it is enough to show that, for any $\varepsilon>0$, there exists some $N\in\mathbb{N}$ and a curve $c=c(\varepsilon)\in\Ind$ such that $s_m\in B_\varepsilon(c)$ for all $n>N$.\par
Given $\varepsilon>0$, define  the curve \[c(\theta)\coloneqq\min\left\{1,\frac{\varepsilon^2}{(4\sqrt{a_0}+4\sqrt{a_1})^2}\right\}\cdot r_0\varphi(\theta),\]  where $a_0,a_1$ are the metric's coefficients as in \eqref{eq: type of metric} and $r_0\in S^{d-1}$ is either the standard basis vector $e_1$ when $d>1$, or ,if $d=1$, $r_0=\pm1$ accordingly to the connected component in which $(s_m)_{m=1}^\infty$ lies in. By \Cref{lemma:shrink cost}, we conclude that $d_G(c,\lambda c)<\frac{\varepsilon}{2}$ for any $0<\lambda\le1$.\par
Next, we choose $N\in\mathbb{N}$ large enough such that $\ell_m<\frac{\ell_c}{9}$ for all $m>N$. Again by \Cref{lemma:shrink cost}, $d_G(s_m,\lambda s_m)<\frac{\varepsilon}{6}$ for any $0<\lambda\le1$.  For each $m>N$, define
\[\lambda_m\coloneqq\min\left\{1, \frac{\varepsilon^2}{6^2(a_0+a_1)C^2}, \frac{\varepsilon^2}{6^2|v_m|^2a_0}\right\}\] where  $C$ is the constant from \Cref{lemma: rot cost}. \par
Notice $\lambda_m\le1$ and particularly $d_G(s_m,\lambda_m s_m)<\frac{\varepsilon}{6}$. In addition, the curve length of $\lambda_m s_m$ is at most $\frac{\varepsilon^2}{6^2|v_m|^2a_0}$, thus, \Cref{lemma: translation cost} implies \[d_G\big(\lambda_m s_m, \lambda_m s_m - \lambda_mv_m\big)<\frac{\varepsilon}{6}.\] 
Next, since $\big(\lambda_ms_m -\lambda_mv_m\big) (\theta)= \lambda_m r_m \ell_m \varphi(\theta)$ we conclude by \Cref{lemma: rot cost} 
\[d_G\left( \lambda_m r_m \ell_m \varphi,  \lambda_m r_0 \ell_m \varphi\right) \le C \sqrt{G_{\lambda_m r_m \ell_m \varphi}\left(\lambda_m r_m \ell_m \varphi,\lambda_m r_m \ell_m \varphi\right)}.\]
$\lambda_m r_m \ell_m \varphi$ is a straight-line curve, and by  reparametrization invariance we mat assume $\varphi=id_{[0,1]}$. Thus, the terms of order $i\ge2$ of $G_{\lambda_m r_m \ell_m \varphi}\big(\lambda_m r_m \ell_m \varphi,\lambda_m r_m \ell_m \varphi\big)$ vanish, and
\begin{equation*}
\begin{split}
    d_G\left( \lambda_m r_m \ell_m \varphi,  \lambda_m r_0 \ell_m \varphi\right) \le& C \sqrt{\sum_{0\le i\le n}a_i\|\nabla^i_{\partial_s}\lambda_m r_m \ell_m \varphi\|_{L^2(ds)}^2} \\
    =&C\sqrt{a_0\frac{(\lambda_m\ell_m)^3}{3}+a_1\lambda_m\ell_m} \le C\sqrt{(a_0+a_1)\lambda_m},
    \end{split}
\end{equation*}
which by choice of $\lambda_m$ implies that $d_G\left( \lambda_m r_m \ell_m \varphi,  \lambda_m r_0 \ell_m \varphi\right) \le \frac{\varepsilon}{6}$. Overall we conclude by triangle inequality that, for any $m>N$,  
\[d_G(c,s_m)\le d_G(c,\lambda_m r_0 \ell_m \varphi)+d_G(\lambda_m r_0 \ell_m \varphi,s_m)<\frac{\varepsilon}{2}+3\frac{\varepsilon}{6}=\varepsilon.\]
\end{proof}

 The proof of \Cref{claim: straight-line Cauchy} implies that the limit point of a diverging straight-line Cauchy sequence may only depend on the parametrizations of the curves, as detailed below:
\begin{cor}
\label{cor: rotation-translation invariance}
    The limit point in $\overline{\Ind}$ of any diverging straight-line Cauchy sequence
    \[ s_m=v_m+r_m\ell_m\varphi_m\in\Ind\] 
    is independent of the choice of $\big{(}r_m\big{)}_{m=1}^\infty\subset S^{d-1}$,$\big{(}v_m\big{)}_{m=1}^\infty\subset\mathbb{R}^d$, and $\big{(}\ell_m\big{)}_{m=1}^\infty\subset\mathbb{R}_+$ such that $\ell_m\searrow0$.
\end{cor} 
\begin{proof}
    We begin by considering the case of a uniformly parametrized sequence, i.e.,  $\varphi_m=\varphi\in\diffp$ for all $m\in\mathbb{N}$. For any two sequences 
    \[s_m'\coloneqq v_m'+r_m'\ell_m'\varphi,\quad s_m\coloneqq v_m+r_m\ell_m\varphi \qquad \forall m\in\mathbb{N},\]
    the sequence
    \[t_m\coloneqq
        \begin{cases}
            s_{\frac{m}{2}}',& m\equiv0\mod2\\
            s_{\frac{m+1}{2}},& m\equiv1\mod2
        \end{cases}\]
     is also Cauchy by \Cref{claim: straight-line Cauchy}. In particular,
    \[d_G\big{(}s_m',s_m\big{)}=d_G\big{(}t_{2m-1},t_{2m}\big{)}\underset{m\rightarrow\infty}{\longrightarrow}0,\]
    which implies $\big{(}s_m'\big{)}_{m=1}^\infty$ and $\big{(}s_m\big{)}_{m=1}^\infty$ share the same limit point in the metric completion of $ \Ind$.\par
    For the general case  
    \[s_m'\coloneqq v_m'+r_m'\ell_m'\varphi_m,\quad s_m\coloneqq v_m+r_m\ell_m\varphi_m \qquad \forall m\in\mathbb{N},\] 
    by the reparametrization invariance, we conclude that 
    \[d_G\big(s_m,s_m')=d_G\big(v_m+r_m\ell_m\varphi,v_m'+r_m'\ell_m'\varphi\big)\underset{m\rightarrow\infty}{\longrightarrow}0,\]
    where $\varphi\in\diff$ is any diffeomorphism.
\end{proof}

If \Cref{conj: Thesis question} were to hold, it would imply that any pair of vanishing length straight-line Cauchy sequences converge to the same limit point in the completion space. The question of whether this is the case is equivalent, by \Cref{cor: rotation-translation invariance} and the reparametrization invariance of $d_G$, to the question whether two sequences of the form $e_1\ell_m\varphi(\theta)$, and $e_1\ell_m\psi(\theta)$ for $\varphi\ne\psi\in\diffp$ converge to the same limit point in the metric completion space. This question is examined in one-dimensional settings $\In$ in the following section.

\subsection{ The One-Dimensional Case}
\label{subsec:one dimensional resluts}

  In this section, we focus on the space $\In$, as a way to isolate the properties of straight line sequences. When considering a curve $c\in\In$, since $c'(\theta)\ne0$ for any $\theta\in[0,1]$, $c$ is either strictly monotone increasing or decreasing. This allows for a useful characterization of $\In$ we will use throughout this section:
\begin{prop}
    Any curve $c\in\In$ can be described as 
    \[c=v+\ell\varphi\] 
    meaning $c(\theta)=v+\ell\varphi(\theta)$, where $v\in\mathbb{R}$, $\ell\in\mathbb{R}_+$ and $\varphi\in\diff$.
\end{prop}

By \cite[Theorem~6.3]{Bauer2020}, any diverging Cauchy sequence in $\Ind$, for any $d\ge1$,  must be of vanishing length. This implies, in the case of $\In$:
\begin{prop}
\label{prop: div cauchy seq}
    In  $\In$, any diverging Cauchy sequence $(c_m)_{m=1}^\infty$ can be written as
    \[c_m=v_m+\ell_m\varphi_m\] where $\mathbb{R}_+\supset\ell_m\rightarrow0$,  $v_m\in\mathbb{R}$, and $\varphi_m\in\diff$ for all $m\in\mathbb{N}$.
\end{prop}
Notice that the direction of a curve $c\in\In$, it being increasing or decreasing, is embodied in the choice of the diffeomorphism $\varphi$ with it being either endpoint-fixing or endpoint-switching, respectively, as defined in \Cref{def: diffeomorphism}. \par

For a pair of curves $c_1=v_1+\ell_1\varphi_1$ and $c_2=v_2+\ell_2\varphi_2$ such that $\varphi_1\in\diffp$ and $\varphi_2\in\diffm$, it is easy to notice that there is no path of immersion connecting them. This is due to the fact that any such path $c:[0,1]\rightarrow\In$ must have, for every $\theta\in[0,1]$, some $t\in(0,1)$ such that $\partial_\theta c(t,\theta)=0$. Thus, at least for some $t\in[0,1]$, the curve $c(t)$ is not an immersion.\par
This means, by the definition of the geodesic metric, that the space $\In$ consists of two distinct connected components:
\[\bigg\{c=v+\ell\varphi \ \bigg| \ v\in\mathbb{R},\ell\in\mathbb{R}_+,\varphi\in\diffp \bigg\} \] and 
\[\bigg\{c=v+\ell\varphi \ \bigg| \ v\in\mathbb{R},\ell\in\mathbb{R}_+,\varphi\in\diffm \bigg\},\]

This is unlike the case of $d>1$, where the space $\Ind$ is connected as shown by \Cref{cor: single con com} below.\par
Naturally, we will restrict our analysis to a single connected component, addressing only endpoint-fixing diffeomorphisms.
It can be verified, by considering the isometry \[v+\ell\varphi(\theta) \mapsto v+\ell\varphi(1-\theta),\] 
that the results are symmetric for the other connected component. \par
By \Cref{claim: straight-line Cauchy}, we conclude that any sequence of the form 
\[ c_m=v_m+\ell_m\varphi,\qquad  \ell_m\searrow0,\ v_m\in\mathbb{R},\ \varphi\in\diff\]
is a diverging Cauchy sequence. Furthermore, \Cref{cor: rotation-translation invariance} ensures that the limit point in the completion space is independent of the choice of the sequences $(v_m)_{m=1}^\infty,(\ell_m)_{m=1}^\infty$. Thus, we may assume that $v_m=0$ for all $m\in\mathbb{N}$ and consider some fixed sequence $\ell_m\rightarrow0$ throughout the analysis.
\begin{remark*}
    When restricting ourselves to the case of $d=1$, the rotation invariance of \Cref{claim: straight-line Cauchy} becomes trivial since $SO(1)=\{1\}$.
\end{remark*}

\subsubsection{Distinct Limit Points in the Metric Completion}

We now turn to the question that motivated this one-dimensional viewpoint, is it possible, for two distinct diffeomorphisms $\varphi,\psi\in\diff$, and a sequence $\ell_m\searrow0$, that the sequences $\big{(}\ell_m\varphi\big{)}_{m=1}^\infty$ and $\big{(}\ell_m\psi\big{)}_{m=1}^\infty$ converge to the same limit point in the metric completion of $\In$? In other words, is it possible that
\[d_G\big{(}\ell_m\varphi,\ell_m\psi\big{)}\underset{m\rightarrow\infty}{\longrightarrow}0\, ?\]

We begin by defining:
\begin{definition}
    \label{def: delta phipsi}
    For any pair of diffeomorphisms $\varphi,\psi\in\diffp$, define 
    \begin{equation}
        \label{eq: delta def} \Delta(\varphi,\psi)\coloneqq\max_{\theta_0,\theta_1\in[0,1]}\bigg|\ln
        \Big(\frac{\varphi'(\theta_1)}{\psi'(\theta_1)}\Big) -\ln\Big(\frac{\varphi'(\theta_0)}{\psi'(\theta_0)}\Big)\bigg|.
    \end{equation}
\end{definition}
Notice that by the mean value theorem, for any pair  $\varphi,\psi\in\diffp$, there exists a point $\theta\in[0,1]$  such that $\varphi'(\theta)=\psi'(\theta)$. By fixing $\theta_0$ to be as such, and maximizing over $\theta_1\in[0,1]$ we conclude:
\begin{prop}
    \label{prop: delta positive}
    For any pair of distinct diffeomorphisms $\varphi\ne\psi\in\diffp$, \[\Delta(\varphi,\psi)>0.\]
\end{prop}

The quantity $\Delta(\varphi,\psi)$ provides the following estimate for the length of paths in $\Gamma(\lambda\psi,\lambda\varphi)$:

\begin{lemma}
\label{lemma: dist-len ineq.}
     Let $n\ge2$, and let $G$ be a Riemannian metric as in \eqref{eq: type of metric}  on $\In$ such that $a_2>0$. Then for every pair $\psi\ne\varphi\in\diffp$ of distinct diffeomorphisms, and every $\gamma\in\Gamma\left(\lambda\varphi,\lambda\psi\right)$, we have  
    \begin{equation}
    \label{eq: delta bound}
        \len(\gamma)\ge\Delta(\varphi,\psi)\sqrt{\frac{a_2}{\max_{t\in[0,1]}\ell_{\gamma(t)}}}.
    \end{equation}
\end{lemma}
\begin{remark*}
The only dependence on $\lambda$ in the right-hand side of \eqref{eq: delta bound} is through $\max_{t\in[0,1]}\ell_{\gamma(t)}$.
\end{remark*}
\begin{proof}
    By definition, the length of $\gamma$ is given by 
    \begin{equation}
    \label{lemeq:length}
        \begin{split}
            \len(\gamma)=&\int_0^1\sqrt{G_c\big{(}\gamma_t,\gamma_t\big{)}}dt=\int_0^1\sqrt{\sum_{0\le i\le n} a_i\|\nabla^i_{\partial_s}\gamma_t\|_{L^2(ds)}^2}dt\\
            \ge&\int_0^1\sqrt{a_2\|\nabla^2_{\partial_s}\gamma_t\|_{L^2(ds)}^2}dt=\sqrt{a_2}\int_0^1\|\nabla^2_{\partial_s}\gamma_t\|_{L^2(ds)}dt.
        \end{split}
    \end{equation}
    Since $\gamma$ remains in the endpoint-fixing connected component, $\gamma'(t,\theta)=\partial_\theta\gamma(t,\theta)>0$ for every $\theta\in[0,1]$ and $t\in[0,1]$, and 
    \[\nabla^2_{\partial_s}\gamma_t=\frac{1}{\gamma'}\partial_\theta\frac{\gamma_t'(\theta)}{\gamma'}.\]
    Moreover, by Cauchy-Schwartz inequality for $G_{\gamma(t)}$, \[\frac{\|\nabla^2_{\partial_s}\gamma_t\|_{L^1(ds)}}{\sqrt{\ell_{\gamma(t)}}}=\frac{\|\nabla^2_{\partial_s}\gamma_t\|_{L^1(ds)}}{\|1\|_{L^2(ds)}}
    \le\|\nabla^2_{\partial_s}\gamma_t\|_{L^2(ds)}.\] 
    Applying this in \eqref{lemeq:length} point-wise for every $t\in[0,1]$ results in
    \begin{equation}
    \label{eq: delta proof1}
        \len(\gamma)\ge\sqrt{a_2}\intop_0^1\frac{\|\nabla^2_{\partial_s}\gamma_t\|_{L^1(ds)}}{\sqrt{\ell_{\gamma(t)}}}dt\ge\sqrt{\frac{a_2}{\max_{t\in[0,1]}\ell_{\gamma(t)}}}\intop_0^1\|\nabla^2_{\partial_s}\gamma_t\|_{L^1(ds)}dt.
    \end{equation}
    For any $\theta_0\le\theta_1\in[0,1]$, we have the following inequality,
    \begin{equation*}
        \begin{split}
            \|\nabla^2_{\partial_s}\gamma_t\|_{L^1(ds)}&\coloneqq\int_0^1|\frac{1}{\gamma'}(\partial_\theta\nabla_{\partial_s}\gamma_t)|\cdot|\gamma'|d\theta=\int_0^1|\partial_\theta\nabla_{\partial_s}\gamma_t|d\theta\\
            &\ge\int_{\theta_0}^{\theta_1}|\partial_\theta\nabla_{\partial_s}\gamma_t|d\theta\ge\bigg|\int_{\theta_0}^{\theta_1}\partial_\theta\nabla_{\partial_s}\gamma_td\theta\bigg|.
        \end{split}
    \end{equation*}
    By the fundamental theorem, for every $t\in[0,1]$,
    \begin{equation*}
        \begin{split}       
            \bigg|\int_{\theta_0}^{\theta_1}\partial_\theta\nabla_{\partial_s}\gamma_t(t,\theta)d\theta\bigg|=&\bigg|\nabla_{\partial_s}\gamma_t(t,\theta_1)-\nabla_{\partial_s}\gamma_t(t,\theta_0)\bigg|\\
            =&\bigg|\frac{1}{\gamma'}\partial_\theta\partial_t\gamma(t,\theta_1)-\frac{1}{\gamma'}\partial_\theta\partial_t\gamma(t,\theta_0)\bigg|.
        \end{split}
    \end{equation*}

    Thus, by \eqref{eq: delta proof1} and the integral triangle inequality we have
    \begin{equation}
    \label{eq: delta proof2}
        \begin{split}
            \len(\gamma)\ge\sqrt{\frac{a_2}{\max_{t\in[0,1]}\ell_{\gamma(t)}}}\cdot\big|\intop_{\theta_0}^{\theta_1}\left(\frac{1}{\gamma'}\partial_\theta\partial_t\gamma(t,\theta_1)-\frac{1}{\gamma'}\partial_\theta\partial_t\gamma(t,\theta_0)\right)dt\big|.
        \end{split}
    \end{equation} 
    Although changing the order of derivatives may seem trivial when considering $\gamma:[0,1]\times[0,1]\rightarrow\mathbb{R}$, it is not the case as we are only guaranteed that $\gamma$ is once differentiable with respect to $t$.  However,   \cite[Theorem~9.41]{alma990011799400203701} implies that if $\gamma:(0,1)\times(0,1)\rightarrow\mathbb{R}^d$ such that $\partial_t\gamma,\partial_\theta\gamma, \text{ and } \partial_{\theta}\partial_t\gamma$ exist, and $\partial_{\theta}\partial_t\gamma$ is continuous, then  $\partial_{t}\partial_\theta\gamma$ exists and 
    \[\partial_{\theta}\partial_t\gamma=\partial_{t}\partial_\theta\gamma.\]
    Therefore,
    \[ \frac{1}{\gamma'}\partial_\theta\partial_t\gamma(t,\theta)=\frac{\partial_t\gamma'}{\gamma'}(t,\theta)=\partial_t\big(\ln(\gamma')\big),\qquad \forall \theta\in[0,1], \ \forall t\in[0,1].\]
    This, together with the use of the fundamental theorem on \eqref{eq: delta proof2} shows:
    \[\len(\gamma)\ge \sqrt{\frac{a_2}{\max_{t\in[0,1]}\ell_{\gamma(t)}}}\bigg|\ln\left(\frac{\gamma'(1,\theta_1)}{\gamma'(0,\theta_1)}\right)-\ln\left(\frac{\gamma'(1,\theta_0)}{\gamma'(0,\theta_0)}\right)\bigg|\]
    where the derivatives in the last expression are taken with respect to $\theta$.\par
    Since $\gamma(0)=\lambda \psi,\gamma(1)=\lambda \varphi$, we have \[\gamma(0,\theta)'=\lambda\psi'(\theta),\ \gamma(1,\theta)'=\lambda\varphi'(\theta) \qquad 
    \forall\theta\in[0,1].\]
   Thus, 
   \[\len(\gamma)\ge \sqrt{\frac{a_2}{\max_{t\in[0,1]}\ell_{\gamma(t)}}}\bigg|\ln\left(\frac{\varphi'(\theta_1)}{\psi'(\theta_1)}\right)-\ln\left(\frac{\varphi'(\theta_0)}{\psi'(\theta_0)}\right)\bigg|.\]
    Taking the maximum on $\theta_0,\theta_1\in[0,1]$ concludes the proof.
\end{proof}
In the following, we use \Cref{lemma: dist-len ineq.} to show that the distance $d_G(\lambda\varphi,\lambda\psi)$ is bounded away from zero uniformly in $\lambda$. Notice that the right-hand side in \eqref{eq: delta bound} is not necessarily bounded away from zero for every path, but by focusing on paths $\gamma^\lambda\in\Gamma\left(\lambda\varphi,\lambda\psi\right)$ that approximate $d_G(\lambda\varphi,\lambda\psi)$, we can ensure that $\max_{t\in[0,1]}\ell_{\gamma^\lambda(t)}$ is bounded from above uniformly in $\lambda$.

\begin{theorem}
    \label{thm: distinct limit points}
    Let $n\ge2$, and let $G$ be a constant coefficient Riemannian metric as in \eqref{eq: type of metric} on $\In$ such that $a_2>0$. Then for any pair $\psi\ne\varphi\in\diffp$ of distinct diffeomorphisms, and any sequence $(\ell_m)_{m=1}^\infty\subset\mathbb{R}_+$ such that $\ell_m\rightarrow0$, there exists some $\delta>0$ such that
    \begin{equation}
        \label{eq: thm dist}
        d_G\big{(}\ell_m\psi,\ell_m\varphi\big{)}>\delta.
    \end{equation}
    In particular, the Cauchy sequences $\big{(}\ell_m\psi\big{)}_{m=1}^\infty$ and $\big{(}\ell_m\varphi\big{)}_{m=1}^\infty$ converge to distinct points in the metric completion of $\In$.
\end{theorem}

\begin{proof}
    By \Cref{claim: straight-line Cauchy}, the sequences $\big{(}\ell_m\psi\big{)}_{m=1}^\infty$ and $\big{(}\ell_m\varphi\big{)}_{m=1}^\infty$ are indeed Cauchy, thus the limit 
    \[\lim_{m\rightarrow\infty}d_G\big{(}\ell_m\psi,\ell_m\varphi\big{)}\]exists and is finite. In particular, $ d_G\big(\ell_m\psi,\ell_m\varphi\big)$is bounded from above by some $D>0$.
    Next, assume $\gamma_m\in\Gamma\left(\ell_m\varphi,\ell_m\psi\right)$, such that 
    \[\len(\gamma_m)<2d_G\big{(}\ell_m\psi,\ell_m\varphi\big{)}\le2D.\]
    By \cref{lemma: dist-len ineq.},  
    \[d_G\big{(}\ell_m\psi,\ell_m\varphi\big{)}>\frac{\len(\gamma_m)}{2}\ge\frac{\Delta(\varphi,\psi)}{2}\sqrt{\frac{a_2}{\max_{t\in[0,1]}\ell_{\gamma_m(t)}}}\]
    On the right hand side, both $\Delta(\varphi,\psi)$ and $\sqrt{a_2}$ are constants independent of $m$. Since $\big{(}\ell_m\psi\big{)}_{m=1}^\infty$ is Cauchy, it is bounded in some ball $B_{r_0}(c_0)\subset\In$. Taking $r=r_0+2D$ ensures that for all $m\in\mathbb{N}$, the path $\gamma_m$ is contained in $B_r(c_0)$. As stated in \Cref{cor: boundness of ellc}, the function $c\mapsto\ell_c$ is bounded on $B_r(c_0)$. Hence, there exists some uniform bound, $L>0$, such that
    \[\ell_{\gamma_m(t)}<L \quad \forall m\in\mathbb{N},\quad \forall t\in[0,1]\]
    We conclude that 
    \begin{equation}
        \label{eq: dist delta} d_G\big{(}\ell_m\psi,\ell_m\varphi\big{)}\ge\frac{\Delta\left(\varphi,\psi\right)\sqrt{a_2}}{2\sqrt{L}}\eqqcolon\delta>0
    \end{equation}     
    In particular, the limit points of $\big{(}\ell_m\psi\big{)}_{m=1}^\infty$ and of $\big{(}\ell_m\varphi\big{)}_{m=1}^\infty$ in the metric completion space are distinct.
\end{proof}

\begin{definition}
    We denote the metric completion space of $\Big(\In,d_G\Big)$ by 
        \[\overline{\In}.\] 
    We denote the induced metric by $\bar d_G$.
\end{definition}
The result of \Cref{thm: distinct limit points} can now be stated as follows:

\begin{cor}
\label{cor: comp points}
Let $n\ge2$, let $G$ be a constant coefficient Riemannian metric as in \eqref{eq: type of metric} on $\In$ such that $a_2>0$, and let $\ell_m\searrow0$ be some sequence. Denote by $\varphi^0$ the limit point in $\overline{\In}$ of $(\ell_m\varphi)_{m=1}^\infty$ for $\varphi\in\diff$ .\par
    Then, the mapping  
    \begin{equation*}
    \begin{split}
        \diff&\hookrightarrow\overline{\In}\setminus\In\\
        \varphi&\mapsto\varphi^0
    \end{split}
    \end{equation*} 
    is injective, and we define 
    \[\diffO\coloneqq\bigg\{\varphi^0\bigg|\varphi\in\diff\bigg\}.\]
\end{cor}
\begin{proof}
    \Cref{cor: rotation-translation invariance} ensures that the mapping described is independent of the choice of $(\ell_m)_{m=1}^\infty$  and thus well defined. \Cref{claim: straight-line Cauchy} guarantees that $\varphi^0\in\overline{
    \In}\setminus\In$. The fact that the mapping $\varphi\mapsto\varphi^0$ is injective is the result of \Cref{thm: distinct limit points}.
\end{proof}

\subsubsection{ Characterization of the Completion space}
\Cref{cor: comp points} established, in the case of a constant coefficient Riemannian metric $G$ with $a_2>0$, the inclusion
\[\diffO\subset\overline{\In}\setminus\In.\]
As a subset of $\overline{\In}$, it can be thought of as a metric space with the metric induced by $\bar d_G$.\par
In the following, we look into the properties of $\diffO\subset\overline{\In}$. We begin with the following observation: 
\begin{lemma}
\label{lem: comp char}
    Any point in $\overline{\In}\setminus\In$ can be expressed as a limit of a sequence from $\diffO$.\par
    In other words, we have  
    \[\overline{\In}=\In\cup\overline{\diffO}\]
    where the closure $\overline{\diffO}$ is taken with respect to the metric $ \bar d_G$.
\end{lemma}
\begin{remark*}
    Note that this characterization describes the completion of the whole space, i.e., both connected components. Furthermore, it holds for any constant coefficient metric $G$ and does not necessarily require $a_2>0$. However, in the case of $a_2=0$, the mapping in \Cref{cor: comp points} may not be injective, and thus $\diffO$ may only correspond to a subset of $\diff$, possibly to a single element.
\end{remark*}
\begin{proof}
    Let $\eta\in\overline{\In}\setminus\In$, by \Cref{prop: div cauchy seq} it must be the limit of some Cauchy sequence $c_m=v_m+\ell_m\varphi_m$, and by \Cref{cor: rotation-translation invariance}, we may assume $v_m=0$. \par
    Now consider $\bar d_G (\ell_m\varphi_m,\varphi_m^0)$ for any fixed $m\in\mathbb{N}$. Since $\varphi_m^0$ is the limit of any Cauchy sequence of the form $(\lambda_k\varphi_m)_{k=1}^\infty$ with $\lambda_k\searrow0$,  by definition of the induced metric,
    \[\bar d_G (\ell_m\varphi_m,\varphi_m^0)\coloneqq\lim_{k\rightarrow\infty}d_G(\ell_m\varphi_m,\lambda_k\varphi_m).\]
    For $k$ large enough, $\lambda_k<\ell_m$, and by \Cref{cor: shrink dist} we have
    \[d_G(\ell_m\varphi_m,\lambda_k\varphi_m)\le C\max(\ell_m^{1.5},\ell_m^{0.5})\]
    Meaning
    \[\bar d_G(\ell_m\varphi_m,\varphi_m^0)<C\max(\ell_m^{1.5},\ell_m^{0.5})\underset{m\rightarrow\infty}{\longrightarrow}0.\]
    By triangle inequality, we have
    \[\bar d_G(\eta,\varphi_m^0)\le\bar d_G(\eta,\ell_m\varphi_m)+\bar d_G(\ell_m\varphi_m,\varphi_m^0).\]
    Since $d_G(\eta,\ell_m\varphi_m)\rightarrow0$ by the definition of $\eta$, and $\bar d_G(\ell_m\varphi_m,\varphi_m^0)\rightarrow0$ as we have shown,  \[\eta=\lim_{m\rightarrow\infty}\varphi_m^0\] which concludes the proof.
\end{proof}
Next, we aim to provide a more detailed characterization of $\overline{\diff}$, for the case of a metric $G$ with $a_2>0$. 

\begin{definition}
    The space $C\big([0,1]\big)$ is defined as the space of all continuous functions $[0,1]\rightarrow\mathbb{R}$, endowed with the supremum norm $\|f\|_\infty\coloneqq\sup_{\theta\in[0,1]}|f(\theta)|$.\par
    The space $C^1\big([0,1]\big)$ is defined as the space of all continuously differentiable functions $[0,1]\rightarrow\mathbb{R}$, endowed with the norm 
    \begin{equation}
    \label{eq: C1 metric}
        \|f\|_{C^1}\coloneqq\|f\|_\infty+\|f'\|_\infty.
    \end{equation}
    
\end{definition}
It is well known that $C^1\big([0,1]\big)$ is a Banach space; thus, complete as a metric space.
Notice the inclusion 
\[\diff\subset\mathcal{D}^1\big([0,1]\big)  \subset C^1\big([0,1]\big)\subset C\big([0,1]\big).\]\par
Prior to the analysis of diffeomorphism sequences, we collect the following lemma:
\begin{lemma}
\label{lemma: exp Cauchy}
    Let $(f_m)_{m=1}^\infty\subset C\big([0,1]\big)$ be a convergent sequence, then the sequence $(g_m)_{m=1}^\infty\subset C\big([0,1]\big)$ defined as $g_m(\theta)\coloneqq e^{f_m(\theta)}$ is convergent as well. 
\end{lemma}
\begin{proof}
    Denote by $f\in C^0 \big([0,1]\big)$ the limit of $(f_m)_{m=1}^\infty$, and consider \[\big|g_m(\theta)-e^{f(\theta)}\big|=|e^{f(\theta)}|\cdot\big|\frac{e^{f_m(\theta)}}{e^{f(\theta)}}-1\big|\le\|e^{f(\theta)}\|_\infty\cdot\big|\frac{e^{f_m(\theta)}}{e^{f(\theta)}}-1\big|.\]
    Since $\big|f_m(\theta)-f(\theta)\big|\underset{m\rightarrow\infty}{\longrightarrow}0$ uniformly in $\theta$, 
    \[e^{f_m(\theta)-f(\theta)}\underset{m\rightarrow\infty}{\longrightarrow}1\]
    uniformly in $\theta$, thus concluding the proof. 
\end{proof}

We now proceed to prove that any element in $\overline{\In} \setminus \In$ can be uniquely identified with a diffeomorphism $\eta\in\mathcal{D}^1\big([0,1]\big)$:

\begin{theorem}
    Let $G$ be a constant coefficient Riemannian metric as in \eqref{eq: type of metric} on $\In$, where $n\ge2$, such that $a_2>0$, and let $(\varphi_m^0)_{m=1}^\infty\subset \diffO\subset\overline{\In}$ be a Cauchy sequence. Then the corresponding sequence 
    \[(\varphi_m)_{m=1}^\infty\subset\diff\subset C^1\big([0,1]\big)\]
    is convergent w.r.t\ $\|\cdot\|_{C^1}$. Moreover, the limit $\varphi\coloneqq\lim_{m\rightarrow\infty}\varphi_m$ is a $C^1$-diffeomorphism of the interval $[0,1]$ onto itself.
\end{theorem}

\begin{proof}
Let $(\varphi_m^0)_{m=1}^\infty\subset\diffO$ be Cauchy with respect to $\bar d_G$. Our goal is to use the inequality from \Cref{thm: distinct limit points}. 
To this end, we need to relate the Cauchy condition on $(\varphi_m^0)_{m=1}^\infty$ to an appropriate condition on a family of curves in $\In$. 
This will enable us to relate $\bar d_G (\varphi_m^0,\varphi_k^0)$ to the length of some path $\gamma:[0,1]\rightarrow\In$ to which \Cref{thm: distinct limit points} applies. \par 
Denote by $C$ the constant from \Cref{lemma:shrink cost}, and define $\ell_m\coloneqq\frac{1}{m^2\max\{1,C^2\}}$. 
By \Cref{lemma:shrink cost}, 
$\bar d_G\big(\ell_k\varphi_m,\varphi_m^0\big)\le\frac{1}{k}$ for any $m,k\in\mathbb{N}$. Since 
$(\varphi_m^0)_{m=1}^\infty\subset\overline{\In}$ is Cauchy, it is bounded by some ball $B_{r_0}(c_0)$ such that $c_0\in\In$. Taking $r\coloneqq r_0+2$, we have by triangle inequality that $\ell_k\varphi_m\in B_r(c_0)$ for all $m,k\in\mathbb{N}$.\par
Notice that the paths to which \Cref{thm: distinct limit points} applies are paths that connect curves of the same length. Take $4<N_1\in\mathbb{N}$
such that $d_G\big(\varphi_m^0,\varphi_k^0\big)<\frac{1}{4}$ for all $m,k>N_1$. By the triangle inequality,
\[d_G\big(\ell_i\varphi_m,\ell_i,\varphi_m\big)\le2\cdot\frac{1}{\ell_i}+\frac{1}{4}<\frac{1}{2}+\frac{1}{4}<1\qquad \forall i,m,k>N_1.\]
Thus, for $i,m,k>N_1$  and any $\gamma\in\Gamma\big(\ell_i\varphi_m,\ell_i\varphi_k\big)$ such that $\len(\gamma)<2d_G\big(\ell_i\varphi_m,\ell_i\varphi_k\big)$, $\gamma(t)$ remains in $B_r(c_0)$ for any $t\in[0,1]$.\par
As stated in \Cref{cor: boundness of ellc}, $c\mapsto\ell_c$ is bounded on any metric ball in $\Ind$. Thus, we have a uniform bound $L>0$  
\[\ell_{\gamma(t)}<L \quad \forall t\in[0,1]\]
for any $\gamma\in\Gamma\big(\ell_i\varphi_m,\ell_i\varphi_k\big)$ with $\len(\gamma)<2d_G\big(\ell_i\varphi_m,\ell_i\varphi_k\big)$  where $k,m,i>N_1$.\par
Given $\varepsilon>0$, we define $\varepsilon'\coloneqq\varepsilon\sqrt{\frac{a_2}{4L}}$ and take $N>N_1\in\mathbb{N}$ such that 
\[\bar{d}_G\big(\varphi_m^0,\varphi_k^0\big)<\frac{\varepsilon'}{3} \quad \forall m,k>N.\]
Taking $i>\frac{3}{\varepsilon'}$ guarantees, by the triangle inequality, that $d_G\big(\ell_i\varphi_m,\ell_i\varphi_k\big)<3\cdot\frac{\varepsilon'}{3}=\varepsilon'$.
Applying \Cref{thm: distinct limit points}, with $\delta$ chosen as in $\eqref{eq: dist delta}$,  shows that
\[\Delta(\varphi_m,\varphi_k)\le\sqrt{\frac{4L}{a_2}}d_G\big(\ell_k\varphi_m,\ell_k\varphi_k)<\sqrt{\frac{4L}{a_2}}\varepsilon'=\varepsilon\]
for any $m,k >N$, where $\Delta(\varphi_m,\varphi_k)$ is as defined in \Cref{def: delta phipsi}.
Fixing $\theta_0$ such that $\varphi_m'(\theta_0)=\varphi_k'(\theta_0)$ shows that
\begin{equation}
    \begin{split}
    \varepsilon&\ge \Delta(\varphi_m,\varphi_k)\coloneqq\max_{\theta_0,\theta_1}\big|\ln\big(\frac{\varphi_m'(\theta_1)}{\varphi_k'(\theta_1)}\big)-\ln\big(\frac{\varphi_m'(\theta_0)}{\varphi_k'(\theta_0)}\big)\big| \\
    &\ge\max_{\theta\in[0,1]}\big|\ln\big(\frac{\varphi_m'(\theta)}{\varphi_k'(\theta)}\big)\big|=\|\ln(\varphi_m')-\ln(\varphi_k')\|_{\infty}.
    \end{split}
\end{equation}
Since $\varphi'_m$ is bounded from above and away from zero for every $m\in\mathbb{N}$, we have $\|\ln\big(\varphi'_m\big)\|_\infty<\infty$. Thus, we conclude that the sequence $\big(\ln(\varphi_m')\big)_{m=1}^\infty$ is Cauchy as a sequence in $C\big([0,1]\big)$. \par
By completeness of $C\big([0,1]\big)$, $\big(\ln(\varphi_m')\big)_{n=1}^\infty$ is convergent, and by \Cref{lemma: exp Cauchy} we conclude that $(\varphi_m')_{m=1}^\infty$ is convergent in w.r.t.\ $\|\cdot\|_{\infty}$ as well. \par
Notice that for any $\theta\in[0,1]$ we have 
\[\big|\varphi_m(\theta)-\varphi_k(\theta)\big|=\big|\int_0^\theta\varphi_m'(t)-\varphi_k'(t)dt\big|\le\int_0^\theta\|\varphi_m'-\varphi_k'\|_{\infty}dt\le\|\varphi_m'-\varphi_k'\|_{\infty},\]
thus, $(\varphi_m)_{m=1}^\infty$ is Cauchy w.r.t.\ $\|\cdot\|_{\infty}$ and therefore w.r.t.\ $\|\cdot\|_{C^1}$ also. 
Denote $\varphi=\lim_{m\rightarrow\infty}\varphi_m$ the limit in $C^1\big([0,1]\big)$, and notice that \[\varphi_m(0)=0,\quad\varphi_m(1)=1\quad \forall m\in\mathbb{N}\]
implies that $\varphi(0)=0$ and $\varphi(1)=1$. 
Moreover, since $\big(\ln(\varphi_m')\big)_{m=1}^\infty$ is $\|\cdot\|_\infty$ convergent and, in particular, bounded, there is a uniform bound $0<\delta$ such that
\[0<\delta<\varphi_m'(\theta) \quad \forall m\in\mathbb{N}, \  \forall\theta\in[0,1].\]
Since $\varphi'=\lim_{m\rightarrow\infty}\varphi_m'$ with respect to $\|\cdot\|_\infty$, we conclude that also 
\[\varphi'(\theta)\ge\delta>0\quad \forall \theta\in[0,1],\]
thus showing that $\varphi$ is indeed a $C^1$-diffeomorphism of $[0,1]$ onto itself.
\end{proof}

To conclude, the metric completion space $\overline{\In}$ can be written as follows:
\begin{cor}
    \label{cor: final char}
    Let $n\ge2$, and let $G$ be a constant coefficient Riemannian metric as in \eqref{eq: type of metric} on $\In$ with $a_2>0$. Then 
    \[\overline{\In}=\In\cup\overline{\diffO},\]
    where $\overline{\diffO}$ is the closure of $\diffO$ defined in \Cref{cor: comp points}, with respect to the induced metric $\overline{d}_{G}$ on $\overline{\In}$. Moreover, 
    Under the identification 
    \[\diffO\leftrightarrow\diff,\]  defined in \Cref{cor: comp points}, if $\varphi_n\underset{\bar d_G}{\longrightarrow}\psi$ then $\varphi_n\underset{\|\cdot\|_{C^1}}{\longrightarrow}\psi$. Thus,
     $\overline{\diffO}$ corresponds to a set $D$ such that
    \begin{equation*}
        \diff\subseteq D\subseteq\mathcal{D}^1\big([0,1]\big).
    \end{equation*}
    
\end{cor}

\subsection{The Multidimensional Case: A Few Insights}
\label{subsec: extension to multi-dimensional case}
In this section, we return to the multi-dimensional case of $\Ind$ and provide another construction of a path that leaves the space in finite time. This construction is based on the combination of two types of paths with vanishing curve length: shortening and shrinking. \par
Shrinking refers to a path of the type 
\[(0,1]\ni t\mapsto tc(\theta)\qquad c\in\Ind,\]
while by shortening we mean a path of type 
\[(0,1]\ni t\mapsto c(t\theta)\qquad c\in \Ind \cap H^{n+1}\left(S^1,\mathbb{R}^d\right).\]
The reason we require $c\in \Ind \cap H^{n+1}\left(S^1,\mathbb{R}^d\right)$ in the case of shortening is the need for the velocity of curve $\gamma(t) = c(t\theta)$ to be in the correct tangent space.
Having $c\in H^{n+1}\left(S^1,\mathbb{R}^d\right)$ ensure that $\gamma_t\in H^{n}\big((0,1),\mathbb{R}^d\big)$ and thus ensures that $\gamma_t\in T_{\gamma(t)}\Ind$. 
Thus, although $c(t\theta)\in\mathcal{I}^{n+1}\big((0,1),\mathbb{R}^d\big)$ for every $t\in(0,1]$, we consider $t\mapsto c(t\theta)$ as a path in $\Ind$, and measure its length with respect to the corresponding Riemannian metric. \par
As \Cref{lemma:shrink cost} implies, the shrinking of any curve with non-zero curvature has infinite length w.r.t. a Riemannian metric $G$ of type \eqref{eq: type of metric}. In fact, for such a path $\gamma$ we have that the $L^2$-norm of the geodesic curvature satisfies $\|\nabla^2_{\partial_s}\gamma_t\|_{L^2(ds)}=\sqrt{G_{tc}(c,c)}\underset{t\rightarrow0}{\longrightarrow}\infty$. However, we observe that  $G_{c(t\theta)}\big(c(t\theta),c(t\theta)\big)\rightarrow0$ as $t\rightarrow0$, suggesting shrinking the initial curve $c(t\theta)$ becomes less costly as $t\rightarrow0$.\par

 \Cref{claim: Cauchy shortshrink} examines the relationship between shortening and shrinking, and provides a sufficient condition for a path that combines both shrinking and shortening to be of finite length. The Cauchy sequences that arise from this construction share the same limit point as a sequence of constant-speed parametrized straight line curves with vanishing length. \par
 Throughout this section, we use $x\lesssim_{a,b}y$ to indicate that there exists a constant $C>0$ dependent only on $a,b$ such that $x\le C y$. \par

\begin{prop}
\label{claim: Cauchy shortshrink}
    Let $n\ge2$ and $0\le\alpha<\frac{1}{2n-3}$ , and let $c\in\Ind$ such that $c\in H^{n+1}\big((0,1),\mathbb{R}^d\big)$. Then the path 
    \begin{equation}
    \label{eq: shortening shrinking const}
        \begin{split}
            \gamma:(0,1]&\rightarrow\Ind \\
            t&\mapsto t^\alpha c(t\alpha)
        \end{split}
    \end{equation}
    has finite length w.r.t.\ the Riemannian metric\ $G$ on $\Ind$ of type \eqref{eq: type of metric}.
\end{prop}
\begin{proof}
    By reparametrization invariance, we may assume the $c$ is constant-speed. In this case, the partial derivatives of $\gamma$ are 
    \begin{equation*}
        \begin{split}
            \gamma_t(t,\theta)&=\alpha t^{\alpha-1}c(t\theta)+t^\alpha\theta \cdot c'(t\theta)\\
            \gamma_\theta(t,\theta)&=t^{\alpha+1} c'(t\theta),
        \end{split}
    \end{equation*}
    particularly, $|\gamma_\theta|\equiv \ell_ct^{\alpha+1}$ is independent of $\theta$. By the Leibnitz rule, 
    \begin{equation*}
        \begin{split}
        \nabla^k_{\partial_s}\gamma_t(t,\theta)= &\frac{1}{(\ell_c t^{\alpha+1})^k}\nabla^k_{\partial_\theta}\gamma_t(t,\theta)\\
        =&\frac{1}{(\ell_c t^{\alpha+1})^k}\bigg[\big(\alpha +k\big) t^{\alpha-1+k}c^{(k)}(t\theta)+\theta t^{\alpha+k}c^{(k+1)}(t\theta)\bigg],
        \end{split}
    \end{equation*}
    where $c^{(k)}$ is the standard $k$-th derivative of $c$ with respect to $\theta$.
By triangle inequality and the fact that $\theta\le1$, we obtain
\[\|\nabla^k_{\partial_s}\gamma_t\|_{L^2(ds)}^2\lesssim_{\ell_c,k,\alpha} \frac{t^{2\alpha}}{t^{2\alpha k+2}}\|c^{(k)}(t\theta)\sqrt{t^{\alpha+1}}\|_{L^2(d\theta)}^2+\frac{t^{2\alpha}}{t^{2\alpha k}}\| c^{(k+1)}(t\theta)\sqrt{t^{\alpha+1}}\|^2_{L^2(d\theta)}.\]
The term $\sqrt{t^{\alpha+1}}$  appears due to the transition from  $L^2(ds)$ to $L^2(d\theta)$.
Thus,
\begin{equation*}
    \begin{split}
    \label{eq: integral fin}
    \len(\gamma)=&\intop_0^1\sqrt{\sum_{0\le k\le n}a_k\|\nabla^k_{\partial_s}\gamma_t\|_{L^2(ds)}^2}dt\\
    \lesssim&_{\ell_c,n,a_0,..,a_n} \intop_0^1\sqrt{\sum_{0\le k\le n}\bigg(\frac{t^{2\alpha}}{t^{2\alpha k +2}}\|c^{(k)}\sqrt{t^{\alpha+1}}\|_{L^2(d\theta)}^2+\frac{t^{2\alpha}}{t^{2\alpha k}}\|c^{(k+1)}\sqrt{t^{\alpha+1}
    }\|_{L^2(d\theta)}^2\bigg)}dt.
    \end{split}
\end{equation*}
For $i<n+1$ we have $\|c^{(i)}\|_{L^\infty}<\infty$, and, by Hölder inequality,
\[\|c^{(i)}(t\theta)\sqrt{t^{\alpha+1}}\|^2_{L^2(d\theta)}\le\|c^{(i)}\|^2_{L^\infty}\cdot\|\sqrt{t^{\alpha+1}}\|^2_{L^1(d\theta)}=t^{\alpha+1}\|c^{(i)}\|^2_{L^\infty}.\]
While for $i=n+1$, by change of variable,
\begin{equation*}
    \begin{split}
        \|c^{(n+1)}(t\theta)\sqrt{t^{\alpha+1}}\|^2_{L^2(dt\theta)}=&\int_0^1\big|c^{(n+1)}(t\theta)\big|^2t^{\alpha+1}d\theta\\
        =&t^\alpha\int_0^t\big|c^{(n+1)}(\theta)\big|^2d\theta\le t^\alpha\|c^{(n+1)}\|^2_{L^2(d\theta)}.
    \end{split}
\end{equation*}
The inequalities above show that
\[\len(\gamma)\lesssim_{c,n,a_0,...,a_n}\int_0^1\sqrt{\sum_{0\le k\le n}\bigg(\frac{t^{3\alpha+1}}{t^{2\alpha k +2}}+\frac{t^{3\alpha}}{t^{2\alpha k}}\bigg)}dt.\]
By equivalence of the Euclidean $1$-norm and $2$-norm on $\mathbb{R}^{2n}$, it is enough to ensure that
\[\int_0^1\big(t^{3\alpha-2\alpha k-1}\big)^\frac{1}{2}dt<\infty ,\quad\ \int_0^1\big(t^{3\alpha -2\alpha k}\big)^\frac{1}{2}dt<\infty, \qquad  k=0,...,n \]
 to guarantee that $\len(\gamma)<\infty$. Notice the conditions above are satisfied by any $0\le \alpha<\frac{1}{2n-3}$.
\end{proof}
The bound $\alpha<\frac{1}{2n-3}$ in \Cref{claim: Cauchy shortshrink} is sharp, as demonstrated by the following example. Consider $\mathcal{I}^2\left([0,1],\mathbb{R}^2\right)$ endowed with the Riemannian metric \[G_c(h,k)=\sum_{i=0}^2\langle \nabla^i_{\partial_s}h,\nabla^i_{\partial_s}k\rangle_{L^2(ds)},\qquad h,k\in T_c\mathcal{I}^2\left([0,1],\mathbb{R}^2\right),\]
and the path 
\[(0,1]\ni t\mapsto t\big(\cos(t\theta),\sin(t\theta)\big).\]
In this case, $c(\theta)=\big(cos(\theta),\sin(\theta)\big)$ is a smooth curve and $\alpha=1=\frac{1}{4-3}$, yet the length of the path is infinite by a direct calculation. \par
\Cref{claim: Cauchy shortshrink} implies that if $c\in\mathcal{I}^{n+1}\big([0,1],\mathbb{R}^d\big)$, any sequence of the form $\big(t_m^\alpha c(t_m\theta)\big)_{m=1}^\infty$, where $\mathbb{R}_+\ni t_m\rightarrow0$, is a diverging Cauchy sequence in $\Ind$.
For the particular case of $\alpha=0$, we are able to relax the requirement of higher regularity and show the following:

\begin{lemma}
\label{lem: shrort-stright line}
    Let $n\ge2$ and $c\in\Ind$, then 
    \[d_G\big(c(t\theta),tc'(0)\theta\big)\underset{t\rightarrow0}{\longrightarrow}0,\]
    where $tc'(0)\theta$ is a constant-speed straight-line segment curve in the direction of $c'(0)$, and $G$ is a Riemannian metric on $\Ind$ of type \eqref{eq: type of metric}.\par
    In particular, a sequence of the form $\big(c(t_m\theta)\big)_{m=1}^\infty$, where $\mathbb{R}_+\ni t_m\rightarrow0$, converges to the same limit point in $\overline{\Ind}$ as a sequence of a constant-speed vanishing-length straight-line segment curves.
\end{lemma}
\begin{proof}
    We assume that $c$ is a constant-speed parametrized curve, and
    notice that for some $\delta>0$ small enough we have $|c'(\theta)-c'(0)|<\frac{\ell_c}{2}$ for all $\theta<\delta$. For $t<\delta$, we define the linear interpolation path $\gamma\in\Gamma\big(c(t\theta),tc'(0)\theta\big)$ by
    \[[0,1]\ni \tau\mapsto\tau c(t\theta)+(1-\tau)\big(tc'(0)\theta\big).\]
    Notice \[|\gamma_\theta(\tau,\theta)|=t\big|c'(0)+\tau\big(c'(t\theta)-c'(0)
    \big)\big|>t\frac{\ell_c}{2},\quad  \forall\tau\in[0,1], \ \forall\theta\in[0,1],\]
    thus $\gamma$ is well defined for $t<\delta$. \par
    The partial derivative of $\gamma$ with respect to $\tau$ is 
    $\gamma_\tau(\tau,\theta)=c(t\theta)-t\theta c'(0)$.
    It can be verified by writing $\nabla^k_{\partial_s}\gamma_\tau$ explicitly by Faà di Bruno's formula, that, $\nabla^k_{\partial_s}\gamma_\tau$ for $k<n$ consists of a sum of products of $\frac{1}{|c'(0)+\tau(c'(t\theta)-c'(0))|}$ and of derivatives $c^{(i)}$ for $i<k$. Since  $\frac{1}{|c'(0)+\tau(c'(t\theta)-c'(0))|}>\frac{\ell_c}{2}$ and $\|c^{(i)}\|_{L^\infty}<\infty$ for $i=0,..,k$, we conclude
    \[\|\nabla^k_{\partial_s}\gamma_\tau\|_{L^\infty}\le C_1 \qquad k=0,..,n-1,\]
    where $C_1>0$ is a constant independent of $t$ and $\tau$.\par
    For $k=n$, $\nabla^n_{\partial_s}\gamma_\tau$ consists of $\frac{c^{(n)}(t\theta)}{|c'(0)+\tau(c'(t\theta)-c'(0))|^n}$ and a sum of products of $\frac{1}{|c'(0)+\tau(c'(t\theta)-c'(0))|}$ and of derivatives $c^{(i)}$ where $i<n$. By a similar argument and triangle inequality we have 
    \[\|\nabla^n_{\partial_s}\gamma_\tau\|^2_{L^2(ds)}\le 4\Big(\|\frac{c^{(n)}(t\theta)}{|c'(0)+\tau(c'(t\theta)-c'(0))|^n}\|^2_{L^2(ds)}+\|C_2\|^2_{L^2(ds)}\Big)\]
    where $C_2>0$ is again a constant independent of $\tau$ and $t$.
    Similarly to the proof of \Cref{claim: Cauchy shortshrink}, we obtain for $C=\max(C_1,C_2)$
    \begin{equation*}
        \begin{split}
            \|\nabla^k_{\partial_s}\gamma_\tau\|^2_{L^2(ds)}&\lesssim_{C,\ell_c} t,\qquad\qquad k=0,..,n-1\\
            \|\nabla^n_{\partial_s}\gamma_\tau\|^2_{L^2(ds)}&\lesssim_{C,\ell_c} \int_0^t|c^{(n)}(\theta)|^2d\theta+t.
        \end{split}
    \end{equation*}
    Using the equivalence of norms on $\mathbb{R}^n$, we conclude:
    \begin{equation*}
        \begin{split}
            \len(\gamma)=&\int_0^1\sqrt{\sum_{0\le k\le n}a_k \|\nabla^k_{\partial_s}\gamma_\tau\|^2_{L^2(ds)}}d\tau\\
            \lesssim_{C,n,a_0,...,a_n}&\int_0^1\Big(\int_0^t|c^{(n)}(\theta)|^2d\theta\Big)^{\nicefrac{1}{2}}d\tau+\int_0^1\sqrt td\tau\underset{t\rightarrow0}{\longrightarrow}0.
        \end{split}
    \end{equation*}
\end{proof}

The following lemma implies that the Cauchy sequences induced by 
\Cref{claim: Cauchy shortshrink} and by \Cref{lem: shrort-stright line} share the same limit point in $\overline{\Ind}$:

\begin{lemma}
\label{lem: short- shrinkshort}
    Let $n\ge2$ and $0\le\alpha<\frac{1}{2n-3}$, and let $c\in\Ind$ such that also $c\in H^{n+1}\big((0,1),\mathbb{R}^d\big)$. Then
    \[d_G\big(c(t\theta),t^\alpha c(t\theta)\big)\underset{t\rightarrow0}{\longrightarrow}0,\]
    w.r.t.\ the Riemannian metric $G$ on $\Ind$ of type \eqref{eq: type of metric}.
\end{lemma}
\begin{proof}
    We assume $c$ is 
  constant-speed parametrized. Fix $t>0$, and consider the path $\gamma\in\Gamma\big(c(t\theta),t^\alpha c(t\theta)\big)$ defined by
    \[[0,1]\ni\tau\mapsto t^{\alpha\tau}c(t\theta).\]
    The partial derivatives along $\gamma$ are 
    \begin{equation*}
        \begin{split}
            \gamma_\tau(\tau,\theta)&=\ln(t)\alpha t^{\alpha\tau}c(t\theta)\\
            \gamma_\theta(\tau,\theta)&=t^{\alpha\tau+1}c(t\theta)
        \end{split}
    \end{equation*}
    and thus $|\gamma_\theta|\equiv \ell_c t^{\alpha\tau+1}$ is independent of $\theta$, and $\nabla^k_{\partial_s}\gamma_\tau=\frac{1}{\ell_c^kt^{(\alpha\tau+1)k}}\nabla^k_{\partial_\theta}\gamma_\tau$.
    Since $c\in H^{n+1}\big([0,1],\mathbb{R}^d\big)$, we can conclude, similarly to the proof of \Cref{claim: Cauchy shortshrink}, that 
        \[ \|\nabla^k_{\partial_s}\gamma_\tau\|^2_{L^2(ds)}\lesssim_{\ell_c,\alpha,n} \frac{\ln^2(t)}{t^{(2k-3)\alpha\tau-1}}\|c^{(k)}\|_{L^\infty}, \qquad k=0,...,n.
        \]
    Thus,
    \begin{equation*}
        \begin{split}
            \len(\gamma)=\intop_0^1\sqrt{{\sum_{0\le k\le n}}a_k\|\nabla^k_{\partial_s}\gamma_\tau\|^2_{L^2(ds)}}d\tau& \lesssim_{\alpha,c,n,a_0,...,a_n}\intop_0^1\sqrt{\sum_{0\le k \le n}\frac{\ln^2(t)}{t^{(2k-3)\alpha\tau-1}}}d\tau \\
            &\lesssim_n\sum_{0\le k\le n}\intop_0^1|\ln(t)|t^{\frac{1}{2}-\frac{(2k-3)}{2}\alpha\tau}d
            \tau,
        \end{split}
    \end{equation*}
    where we again used the equivalence of norms of norms on $\mathbb{R}^n$.
    Notice 
    \[\int_0^1-\ln(t)t^{\frac{1}{2}-\frac{2k-3}{2}\alpha\tau}d\tau=\frac{2}{\alpha(3-2k)}t^{\frac{1}{2}-\frac{2k-3}{2}\alpha\tau}\bigg|_{\tau=0}^1\] 
    and since $0\le \alpha<\frac{1}{2n-3}$ we have 
    \[\frac{2}{\alpha(3-2k)}\Big(t^{\frac{1}{2}}-t^{\frac{1}{2}-\frac{2k-3}{2}\alpha}\Big)\underset{t\rightarrow0}{\longrightarrow}0.\]
\end{proof}

The combination of \Cref{claim: Cauchy shortshrink},  \Cref{lem: shrort-stright line}, and \Cref{lem: short- shrinkshort} implies:
\begin{cor}
\label{cor: shrinkshort limit point}
    Let $n\ge2$ and $0\le \alpha<\frac{1}{2n-3}$, and let $c\in\Ind$ such that $c\in H^{n+1}\big(0,1),\mathbb{R}^d\big)$. Then any sequence of the form $t_m^\alpha c(t_m\theta)$ where $t_m\rightarrow0$, converges in $\overline{\Ind}$, w.r.t\ the geodesic distance induced by a Riemannian metric of type \eqref{eq: type of metric}, to the same limit point as a vanishing-length constant speed straight-line segment curve sequence.
\end{cor}

In addition, \Cref{claim: Cauchy shortshrink} and \Cref{cor: shrinkshort limit point} also imply:
\begin{cor}
\label{cor: single con com}
    Let $n\ge2$ and $d\ge2$, then $\Ind$ consists of a single connected component. 
\end{cor}
\begin{remark*}
    This is not the case for  $\mathcal{I}^n\big(S^1,\mathbb{R}^2\big)$ with constant coefficient metrics 
    of order $n\ge2$, for example, as shown in \cite[Section~2.9]{michor2006riemanniangeometriesspacesplane}.
    \end{remark*}
\begin{proof}
Locally, the topology of $\Ind$ identifies with that of $H^n\big((0,1),\mathbb{R}^d\big)$, thus, any curve $c\in\Ind$ can be approximated by smooth curves \cite[Theorem~11.24]{leoni2009first}. In particular, $c$ lies in the same connected component as some smooth curve $c_\infty$. By \Cref{cor: shrinkshort limit point}, we conclude that $c_\infty$ lies within the same connected component as some constant-speed straight-line segment curve. The proof of \Cref{claim: straight-line Cauchy} implies that all such segment curves lie within a single connected component. 
\end{proof}

%% file: Discussion.tex
\section{Discussion}
\label{chap: Discussion}
This thesis continues the line of previous studies on the completeness properties of immersed curve spaces presented in \Cref{subsec: Previous Work}. It focuses specifically on the space of Euclidean-valued immersed open curves endowed with constant coefficient reparametrization-invariant Sobolev-type Riemannian metrics, a space known to be metrically incomplete.  The central objective of this thesis is to examine \Cref{conj: Thesis question}, posed initially in \cite[Question~6.4]{Bauer2020}. \par

\Cref{subsec:example analysis} analyzes \cite[Example~6.1]{Bauer2020}, which illustrates the metric incompleteness of $\Ind$ by constructing a path consisting of straight-line segment curves that leaves the space in finite time.  \Cref{lemma:shrink cost}, \Cref{lemma: translation cost}, and \Cref{lemma: rot cost} establish new estimates on the length of paths in $\Ind$, which are then used in \Cref{claim: straight-line Cauchy} to generalize the existing example of diverging Cauchy sequences to a broader family. \par

In \Cref{subsec:one dimensional resluts}, we turn to $\In$, the space of real-valued immersions, as a way to isolate the properties of straight-line segment curves.  There, we prove \Cref{thm: distinct limit points}, which states that if $\In$ is endowed with Riemannian metrics of the form \eqref{eq: type of metric} with $a_2>0$, then the metric completion space contains, at least, an additional copy of $\diff$. Notably, this shows that \Cref{conj: Thesis question} does not hold in this case.  Furthermore, \Cref{cor: final char} provides a nearly complete characterization of the metric structure of the completion space in these settings.\par

The findings in \Cref{subsec:one dimensional resluts} naturally lead to several open questions regarding real-valued immersions: 
\begin{question}
\label{Qs: one dimensional}
Let $n\ge2$ and let $\In$ be endowed with Riemannian metrics $G$ of the form \eqref{eq: type of metric}:
\begin{enumerate}
    \item \label{Q: comp ident}Does the metric completion of  $\mathcal{I}^2\big([0,1],\mathbb{R}\big)$ with respect to $d_{G}$ consist  precisely of $\mathcal{I}^2\big([0,1],\mathbb{R}\big)\cup \mathbf{diff}^0_{H^2}\big([0,1]\big)$?
    \item \label{Q: higher N} Does \Cref{thm: distinct limit points} hold for $n>2$ when the assumption $a_2>0$ is relaxed?
    \item \label{Q: Higher N char}If the answer to \ref{Q: higher N} is positive, how can the metric completion space be characterized? 
\end{enumerate}\end{question}
Answering \Cref{Qs: one dimensional} will essentially lead to a full characterization of the metric completion of $\In$.
It may seem that standard Sobolev inequalities imply that the answer to part \ref{Q: higher N} of \Cref{Qs: one dimensional} is positive. However, this is not the case since  $\|\cdot\|_{L^2(ds)}+\|\cdot\|_{\dot H^n(ds)}$ does not necessarily  dominate the term $\|\cdot\|_{\dot H^2(ds)}$ uniformly on metric balls, as $\ell_c$ in the left hand side of \eqref{eq: Sobolev estimates} may be arbitrary small. \par
It is also worth noting that \Cref{thm: distinct limit points} relates the geodesic distance to \eqref{eq: delta def}, a quantity dependent only on the first derivatives of diffeomorphisms in $\diff$.  The technique used in the proof of \Cref{thm: distinct limit points} depends specifically on second-order derivatives; thus, it might be difficult to generalize to higher-order metrics, as parts \ref{Q: higher N} and \ref{Q: Higher N char} of \Cref{Qs: one dimensional} refer to. \par

That being said, the most significant direction for future research is extending the results of \Cref{subsec:one dimensional resluts} to $\Ind$ with $d > 1$, as these spaces are the primary motivation for studying immersed curve spaces. \Cref{claim: straight-line Cauchy} and \Cref{claim: Cauchy shortshrink} are currently the only constructions of diverging Cauchy sequences in $\Ind$, both resulting in sequences converging to the limit point of vanishing-length straight-line segment sequences.  The main points that remain to be answered are: 
\begin{question} Let $n\ge2$ and $d>1$, and let $\Ind$ be endowed with Riemannian metrics of the form \eqref{eq: type of metric}, 
\begin{enumerate}
\label{Qs: multi dim}
    \item \label{Q: distinct limit}Given $a_2>0$, does the result of \Cref{thm: distinct limit points} extend to $\Ind$, i.e., do any two uniformly-parametrized vanishing-length straight-line segment sequences, whose parametrizations differ, have distinct limit points in the metric completion?
    \item\label{Q: more points} Does $\overline{\Ind}\setminus\Ind$ contain points that are not limit points of vanishing-length straight-line segment sequences?
\end{enumerate}  
\end{question} 
Denote by $G^1,G^d$ the Riemannian metrics of the form \eqref{eq: type of metric} on  $\In$ and $\Ind$ with $d>1$, respectively, with the same sequence of coefficients $a_0,...,a_n$. Any affine embedding $\iota:\mathbb{R}\hookrightarrow\mathbb{R}^d$ induces a map \[\iota^*:\In\hookrightarrow\Ind,\] and for any $\varphi,\psi\in\diff$ and $\lambda\in\mathbb{R_+}$, a map 
\begin{equation}
    \begin{split}
    \Gamma\big(\lambda\varphi,\lambda\psi\big)&\hookrightarrow\Gamma\big(\iota(\lambda\varphi),\iota(\lambda\phi)\big)\\
    \gamma&\mapsto\iota^*\gamma\coloneqq \iota\circ\gamma.
    \end{split}
\end{equation}

\Cref{thm: distinct limit points} implies that  any $\gamma\in\iota^*\Gamma\big(\lambda\varphi,\lambda\psi\big)$ satisfies
\[0<\delta\lesssim_d \len_{G^d}\big(\gamma\big).\] 
By \Cref{lemma: rot cost}, it is possible to conclude that any $\gamma\in\Gamma\big(\iota(\lambda\varphi),\iota(\lambda\psi)\big)$, such that $\gamma(t)$ is a straight-line segment curve for all $t\in[0,1]$, satisfies 
\[0<\delta'\le\len_{G^d}(\gamma)\]
where $0<\delta'$ is some constant that may differ from $\delta$.
However, this is insufficient to answer part \ref{Q: distinct limit} of \Cref{Qs: multi dim} positively, since 
$\Gamma\big(\iota(\lambda\varphi),\iota(\lambda\psi)\big)$ contains paths that do not consist strictly of straight-line curves.  Thus,
\[d_{G^d}\big(\iota(\lambda\varphi),\iota(\lambda\phi)\big)\le d_{G^1}\big(\lambda\varphi,\lambda\phi\big),\]
and it still may be the case that $d_{G^d}\big(\iota(\lambda\varphi),\iota(\lambda\psi)\big)\overset{\lambda\rightarrow0}{\longrightarrow}0$. \par

Another possible approach to part \ref{Q: distinct limit} of \Cref{Qs: multi dim} is adjusting the proof of \Cref{thm: distinct limit points} to the setting of $\Ind$ with $d>1$.  The difficulty in this approach arises from the fact that $\frac{\partial_t\gamma'}{|\gamma'|}=\partial_t\big(\ln(|\gamma'|)\big)$ does not hold for $d>1$. Thus, we are not able to use \eqref{eq: delta proof2} to relate the geodesic distance to a quantity dependent on parametrizations alone. \par

Regarding question \ref{Q: more points}, a possible starting point may be identifying whether a vanishing-length sequence with curvature bounded away from zero may also be a diverging Cauchy sequence.  If so, may it still converge to the same limit point as a vanishing length straight-line segment sequence? 
Such questions remain open even for sequences with unbounded curvature,  as exemplified by the sequence of vanishing circles proposed in \eqref{eq: vanishing circles}. \par 

In light of the findings in \Cref{chap:results} and the challenges in constructing a diverging Cauchy sequence whose curvature remains bounded away from zero, we revise \Cref{conj: Thesis question} as follows:
 \begin{conj}
      \label{conj: updated}
      Let $n\ge2$ and $d\ge1$, and let $\Ind$ be endowed with Riemannian metrics $G$ of the form \eqref{eq: type of metric}. Then the metric completion of $\big(\Ind,d_{G}\Big)$ consists of 
     \[\Ind\cup\diff,\] 
     where each point $\varphi\in \diff$ represents the limit point of a sequence of vanishing-length straight-line segment curves which are $\varphi$-parametrized.
 \end{conj}
If \Cref{conj: updated} holds, it would imply that the metric completion of the corresponding shape space
\[\Ind\slash\diff\]
is given by 
\[\Big(\Ind\slash\diff\Big)\cup\{0\},\]
where $0$ represents the limit point of all vanishing length shape sequences. Essentially, this would deem \Cref{conj: Thesis question} true for the shape space rather than the immersion space.